\documentclass[11pt,a4]{article}
\usepackage{latexsym}
\usepackage{amsthm,amsmath,amssymb,wasysym}
\usepackage{verbatim}
\usepackage[pdftex,colorlinks,citecolor=cyan,linkcolor=magenta,urlcolor=blue,breaklinks=true]{hyperref}
\usepackage{graphicx} 
\usepackage{caption}
\usepackage{subcaption} 
\usepackage{color}

\def\XXint#1#2#3{{\setbox0=\hbox{$#1{#2#3}{\int}$}
     \vcenter{\hbox{$#2#3$}}\kern-.5\wd0}}

\renewcommand{\arraystretch}{1.4}

\newtheorem{Theorem}{Theorem}[section]

\newtheorem{Lemma}[Theorem]{Lemma}

\newtheorem{Proposition}[Theorem]{Proposition}

\theoremstyle{definition}

\newtheorem{Remark}[Theorem]{Remark}

\newtheorem{Example}[Theorem]{Example}

\newtheorem{Assumption}[Theorem]{Assumption}

\def\mathref#1{\ifmmode\mathrm{(\ref{#1})}\else(\ref{#1})\fi}

\def\rhx2{\sqrt{1+ r_{h,x}^2}}

\title{Sufficient conditions for unique global solutions in optimal control of semilinear equations with $C^1-$nonlinearity}  
\author{Ahmad Ahmad Ali\footnote{Schwerpunkt Optimierung und Approximation, 
Universit\"at Hamburg, Bundesstra{\ss}e 55, 20146 Hamburg, Germany.}, Klaus Deckelnick\footnote{Institut f\"ur Analysis und Numerik,
Otto--von--Guericke--Universit\"at Magdeburg, Universit\"atsplatz 2,
39106 Magdeburg, Germany} \& Michael Hinze\footnote{Schwerpunkt Optimierung und Approximation, 
Universit\"at Hamburg, Bundesstra{\ss}e 55, 20146 Hamburg, Germany.}}

\date{}
\graphicspath{{figures/}}

\begin{document}

\maketitle

\begin{center}
 {\it  Dedicated to G\"unter Leugering on the occasion of his 65th birthday.}
\end{center} 

\begin{abstract}
We consider a semilinear elliptic  optimal control problem possibly subject to control and/or state constraints. 
Generalizing previous work in \cite{ali2016global} we provide a condition which guarantees that a solution of the necessary first order conditions is a global minimum.
A similiar result also holds at the discrete level where the corresponding condition can be evaluated explicitly. Our investigations  are motivated by G\"unter Leugering, who raised the question whether the problem class
considered in \cite{ali2016global} can be extended to the nonlinearity $\phi(s)=s|s|$. We develop a corresponding analysis and present several numerical test examples demonstrating  its  usefulness in practice.
\end{abstract}

\section{Introduction and problem setting}
Let $\Omega \subset \mathbb{R}^d \, (d=2,3)$ be a bounded, convex polygonal/polyhedral domain, in which we consider the semilinear elliptic PDE 
\begin{align}
-\Delta y + \phi(\cdot,y) &=u  \quad \mbox{ in }  \Omega, \label{semilinear}\\
y &=0  \quad \mbox{ on } \partial\Omega.  \label{bc}
\end{align}
We assume that $\phi: \bar{\Omega} \times \mathbb{R} \rightarrow \mathbb{R}$ is a Carath\'eodory function with $\phi(x,0)=0$ a.e. in $\Omega$ and that
\begin{eqnarray} 
& & y \mapsto \phi(x,y) \mbox{ is of class } C^1  \mbox{ with } \phi_y(x,y)  \geq  0 \mbox{ for almost all } x \in \Omega; \label{nonneg} \\
& & \forall L \geq 0 \; \exists c_L \geq 0 \quad  \phi_{y}(x,y) \leq   c_L \quad \mbox{ for almost all } x \in \Omega \mbox{ and all } |y | \leq L.  \label{upbound}
\end{eqnarray}
Under the above conditions it can be shown that for every $u \in L^2(\Omega)$  the boundary value problem (\ref{semilinear}), (\ref{bc}) has a unique solution $y=: \mathcal G(u) \in H^2(\Omega) \cap H^1_0(\Omega)$.
Next, let us introduce $U_{ad}:=\{ v \in L^2(\Omega) : u_a \leq v(x) \leq u_b \mbox{ a.e. in } \Omega \}$, where $u_a,u_b \in \mathbb{R}$ with $- \infty \leq u_a \leq u_b \leq \infty$.
For given $y_0 \in L^2(\Omega), \, \alpha >0$ we then consider the optimal control problem
\begin{displaymath}
(\mathbb{P}) \quad
\begin{array}{rcl}
&& \min_{u \in U_{ad}} J(u):=\frac{1}{2} \Vert y-y_0 \Vert_{L^2(\Omega)}^2  + \frac{\alpha}{2} \Vert u \Vert_{L^2(\Omega)} ^2 \\
&& \mbox{subject to } y = \mathcal G(u) \mbox{ and } y_a(x) \leq y(x)\leq y_b(x) \mbox{ for all }  x \in K.
\end{array}
\end{displaymath}
Here, $y_a,y_b \in C^0(\bar \Omega)$  satisfy 
$y_a(x)<y_b(x)$ for all $x \in K$, where $K \subset \bar \Omega$ is compact and  either $K \subset \Omega$ or $K=\bar \Omega$. In
the latter case we suppose in addition that $y_a(x) <0< y_b(x),  x \in \partial \Omega$.
It is well--known that $(\mathbb{P})$ has a solution provided that a feasible point exists (compare \cite{casas1993boundary}). Under some constraint qualification, such as the linearized Slater condition,
a local solution $\bar u \in U_{ad}$ of $(\mathbb{P})$ then satisfies the following necessary first order conditions, see \cite[Theorem~5.2]{casas1993boundary}:  There exist $\bar p \in L^2(\Omega)$ and a regular Borel
measure $\bar \mu \in \mathcal M(K)$ such that 
\begin{align}
&\int_\Omega \nabla \bar y \cdot \nabla v + \phi(\cdot,\bar y)v \, dx = \int_\Omega \bar u v \, dx \quad \forall \, v \in H^1_0(\Omega),  \qquad y_a \leq \bar y \leq y_b \mbox{ in } K, \label{oc:state} \\
&\int_\Omega  \bar p (- \Delta v)  + \phi_y(\cdot,\bar y) \, \bar p \, v \, dx
 =\int_\Omega (\bar y -y_0) v \, dx   +   \int_K v \, d\bar\mu  \quad \forall \,v \in H^1_0(\Omega) \cap H^2(\Omega), \label{oc:adjoint} \\
&\int_\Omega (\bar p + \alpha \bar u)(u-\bar u) \, dx  \geq  0 \qquad \forall \, u \in U_{ad}, \label{oc:VI control}\\
&\int_K (z-\bar y)\, d\bar\mu  \leq  0 \qquad \forall \, z \in C^0(K), y_a \leq z \leq y_b \mbox{ in } K. \label{oc:VI measure}
\end{align}
In view of the nonlinearity of the state equation  problem $(\mathbb{P})$ is in general nonconvex  and hence there may be several solutions of the conditions (\ref{oc:state})--(\ref{oc:VI measure}). The problem we are
interested in is whether it is possible to establish sufficient conditions which guarantee that a solution of (\ref{oc:state})--(\ref{oc:VI measure}) is actually 
a global minimum of $(\mathbb{P})$. A first result in this direction was obtained by the authors in \cite{ali2016global} and holds for a class of nonlinearities which satisfy a certain growth condition:
\begin{Theorem} \label{main0} (\cite[Theorem 3.2]{ali2016global}) Let $d=2$; suppose that  $y \mapsto \phi(x,y)$ belongs to $C^2$ for almost all $x \in \Omega$ and that there exist $r>1$ and $M \geq 0$ such that
 \begin{equation} \label{assumption: 1a}
\displaystyle
|\phi_{yy}(x,y)| \leq M (\phi_y(x,y))^{\frac{1}{r}} \quad \mbox{ for almost all } x \in \Omega \mbox{ and all } y \in \mathbb{R}.
\end{equation}
Assume that $(\bar u, \bar y,\bar p, \bar \mu)$ solves  (\ref{oc:state})--(\ref{oc:VI measure}) and that
\begin{equation}  \label{etadef0}
\displaystyle
\Vert \bar p \Vert_{L^q} \leq \big(\frac{r-1}{2r-1} \big)^{\frac{1-r}{r}} M^{-1}  C^{\frac{2-2r}{r}}_q  \alpha^{\frac{\rho}{2}}    
q^{1/q} r^{1/r} \rho^{\rho/2} (2 - \rho)^{\frac{\rho}{2}-1},
\end{equation}
where $q:=\tfrac{3r-2}{r-1}, \, \rho:=\frac{r+q}{rq}$ and $C_q$ denotes the constant in (\ref{gnineq}) below. Then $\bar u$ is a global minimum for Problem~$(\mathbb{P})$. If the above inequality is strict, then
$\bar u$ is the unique global minimum.
\end{Theorem}
Assumption (\ref{assumption: 1a}) is satisfied for $\phi_q(y):= |y|^{q-2}y$ provided that $q>3$ if we choose $r=\frac{q-2}{q-3}$. G\"unter Leugering recently raised the question whether our theory can be
extended to include the case $q=3$. The corresponding nonlinearity  $\phi_3(y)=|y|y$ appears for example in the mathematical modeling of gas flow through pipes with PDEs \cite[(5.1)]{hante}, so that an extension of Theorem \ref{main0} 
to this case could  be helpful in understanding the optimal control of pipe networks.  As $\phi_3$  is no longer $C^2$ it does not fit directly into the theory above. However it turns out that instead the analysis can be 
built on the fact that $\phi_{3,y}$ satisfies a global Lipschitz condition. \\
The purpose of this paper is to generalize Theorem \ref{main0} in several directions. To begin, we shall replace (\ref{assumption: 1a}) by a condition that can be formulated for $C^1$--nonlinearities $\phi$ and is satisfied
by the functions $\phi_q$ for every $q \geq 3$ thus including the case suggested by G\"unter Leugering, see (\ref{assumption: 1}). A second generalization concerns the choice of the norm $\Vert \bar p \Vert_{L^q}$ in condition (\ref{etadef0}). 
Even though the integration index $q=\frac{3r-2}{r-1}$  is quite natural (solve $r=\frac{q-2}{q-3}$ for $q$), it is
nevertheless possible to formulate a corresponding result not just for one index but for $q$ belonging to  a suitable interval, see (\ref{qcond}), thus giving additional flexibility in its application. 
Our arguments are natural extensions of the analysis presented in \cite{ali2016global}
and will also cover the case $d=3$ left out in Theorem \ref{main0}. \\
There is a lot of literature available considering the problem $(\mathbb{P})$. For a broad overview, we refer the reader to the references of the respective citations. In \cite{casas1993boundary} this problem is studied for boundary controls. The regularity of optimal controls  of $(\mathbb{P})$ and their associated multipliers  is  investigated in \cite{casas2010recent} and  \cite{casas2014new}. Sufficient second order conditions are discussed in e.g. \cite{casas2002second,casas2008necessary,casas2008sufficient} when the set $K$ contains  finitely/infinitely many points. For the role of those conditions in PDE constrained optimization see e.g. \cite{casas2015second}. 

The finite element discretization of problem $(\mathbb{P})$ in rather general settings is studied in \cite{AradaCasasTroeltzsch,casas2002uniform,hinze2012stability}. Convergence rates for sets $K$ containing only finitely many points are  established in  \cite{merino2010error} for finite dimensional controls, and in \cite{casas2002error} for control functions. Only in \cite{neitzel2015finite,ADH2018} an error  analysis is provided for general pointwise state constraints in $K$. Error  analysis for linear-quadratic control problems  can be found in e.g. \cite{casas2014new}, \cite{deckelnick2007convergence,deckelnick2007finite} and \cite{meyer2008error}. Improved error estimates for the state in the case of weakly active state constraints are provided in \cite{neitzel2016wollner}. A detailed discussion of discretization concepts and error analysis in PDE-constrained control problems can be found in \cite{hinze2012discretization,hinze2010discrete} and \cite[Chapter~3]{pinnau2008optimization}.

The organization of the paper is as follows: in \S~\ref{sec:control problem} we shall develop the optimality conditions outlined above. In addition to the criteria based on an $L^q$--norm of $\bar p$ we shall also include
a result that uses a sign of $\bar p$. The variational discretization of $(\mathbb{P})$ is considered in \S~\ref{sec:variational} and is based on a finite element approximation of  (\ref{semilinear}), (\ref{bc})
that uses numerical integration for the nonlinear term. We obtain corresponding optimality criteria for discrete stationary points and apply these conditions in  a series of numerical tests in \S~\ref{numerics}
including the nonlinearity $\phi(y)= y |y|$.

\setcounter{equation}{0}

\section{Optimality conditions for $(\mathbb{P})$}

\label{sec:control problem}

In what follows we assume that $(\bar u, \bar y,\bar p, \bar \mu)$ is a solution of  (\ref{oc:state})--(\ref{oc:VI measure}).
Let $u \in U_{ad}$ be a feasible control,  $y=\mathcal{G}(u)$ the associated state such that $y_a \leq y \leq y_b$ in $K$. A straightforward calculation shows that
\begin{equation} \label{eq1}
J(u)-J(\bar u) = \frac{1}{2} \| y-\bar y \|_{L^2(\Omega)}^2 + \dfrac{\alpha}{2} \| u -\bar u \|_{L^2(\Omega)}^2  + \alpha \int_\Omega \bar u (u-\bar u) \, dx
+ \int_\Omega (\bar y-y_0)(y-\bar y) \, dx.
\end{equation}
Combining  \eqref{oc:adjoint} for $v:= y-\bar y$ with (\ref{oc:VI measure}) and (\ref{semilinear}) we deduce that
\begin{eqnarray*}
\lefteqn{ \hspace{-1.5cm} \int_\Omega (\bar y-y_0)(y-\bar y) \, dx = - \int_{\Omega} \bar p \, \Delta(y - \bar y) \, dx + \int_{\Omega} \phi_y(\cdot,\bar y) \, \bar p \, (y - \bar y) \, dx 
 - \int_K ( y - \bar y) d \bar \mu } \\
& \geq  & \int_{\Omega} (u - \bar u ) \bar p \, dx - \int_{\Omega} \bigl( \phi(\cdot,y) - \phi(\cdot,\bar y) - \phi_y(\cdot,\bar y)(y - \bar y) \bigr) \bar p \, dx.
\end{eqnarray*}
Inserting this relation into (\ref{eq1}) and recalling (\ref{oc:VI control}) we finally obtain
\begin{equation} \label{eq2}
J(u)-J(\bar u)   \geq   \frac{1}{2} \| y-\bar y \|_{L^2(\Omega)}^2 + \dfrac{\alpha}{2} \| u -\bar u\|_{L^2(\Omega)} ^2  - R(u), 
\end{equation}
where  
\begin{equation} \label{remainder}
 R(u) = \int_{\Omega} \bigl( \phi(\cdot,y) - \phi(\cdot,\bar y) - \phi_y(\cdot,\bar y)(y - \bar y) \bigr) \bar p \, dx.
\end{equation}

\subsection{Conditions involving a sign of $\bar p$}

A natural first idea to deduce global optimality from (\ref{eq2}) consists in identifying situations in which $R(u) \leq 0$ for all $u \in U_{ad}$. We
have the following result:

\begin{Theorem} \label{sign}
Suppose that there exists an interval $I \subset \mathbb{R}$ such that $y \mapsto \phi(x,y)$ is convex (concave) on $I$ for almost all $x \in \Omega$.
Furthermore, assume that for every $u \in U_{ad}$ the solution $y=\mathcal G(u)$ with $y_a \leq y \leq y_b$ in $K$ satisfies $y(x) \in I$ for all $x \in \Omega$.
If $\bar p \leq 0 \; ( \bar p \geq 0)$ a.e. on $\Omega$, then $\bar u$ is the unique global minimum of $(\mathbb{P})$.
\end{Theorem}
\begin{proof} Suppose that $y \mapsto \phi(x,y)$ is convex. Then our assumptions imply that
\begin{displaymath}
\phi(x,y(x)) - \phi(x,\bar y(x)) - \phi_y(x,\bar y(x))(y(x) - \bar y(x)) \geq 0 \quad \mbox{ for almost all }  x \in \Omega
\end{displaymath}
which yields that $R(u) \leq 0$ since $\bar p \leq 0$ a.e. in $\Omega$. Hence $J(u)>J(\bar u)$ for $u \neq \bar u$ by (\ref{eq2}).
\end{proof}

In general we cannot expect  the adjoint variable $\bar p$ to have a sign without additional conditions on the data of the problem. The following
result is similar in spirit to a sufficient condition involving a suitable bound on $y_0$ obtained in \cite[Theorem 5.4]{M76} and \cite[Section 5.2]{IK00} for the optimal control of the obstacle problem.

\begin{Lemma} \label{pleq0} 
Suppose that $K= \emptyset$ and that $u_a=0, u_b < \infty$. Let  $y_b \in H^2(\Omega)$ satisfy
\begin{displaymath}
- \Delta y_b + \phi(\cdot,y_b) \geq  u_b \; \mbox{ in } \Omega, \quad y_b \geq 0 \; \mbox{ on } \partial \Omega.
\end{displaymath}
Then $0 \leq \mathcal G(u) \leq y_b$ in $\bar{\Omega}$ for every $u \in U_{ad}$. Also, if $y_0 \geq y_b$ a.e. in $\Omega$, then $\bar p \leq 0$ in $\Omega$.
\end{Lemma}
\begin{proof} Let $u \in U_{ad}$ and set $y=\mathcal G(u)$.  If we test (\ref{oc:state}) with $v= y^-$ we have
\begin{displaymath}
\int_{\Omega} | \nabla y^- |^2 \, dx =- \int_{\Omega} \phi(\cdot,y^-) \, y^- \, dx +  \int_{\Omega} u \, y^- \, dx  \leq 0
\end{displaymath}
using (\ref{nonneg}), the fact that $\phi(\cdot,0)=0$ as well as $u \geq 0$. We infer that $y^- \equiv 0$   and hence $y \geq 0$ in $\bar{\Omega}$. Next, $y - y_b$ satisfies
\begin{displaymath}
-\Delta(y - y_b)  + [\phi(\cdot,y)- \phi(\cdot,y_b)] \leq   u - u_b \leq 0 \quad \mbox{ a.e. in } \Omega.
\end{displaymath}
Testing with $(y - y_b)^+$ then gives $y \leq y_b$ in $\bar{\Omega}$. Finally, since $K= \emptyset$,
the adjoint state satisfies
\begin{displaymath}
- \Delta \bar p + \phi_y(\cdot,\bar y) \bar p = \bar y - y_0 \leq y_b - y_0 \leq 0 \quad \mbox{ a.e. in } \Omega
\end{displaymath}
since $\bar y \leq y_b$ by what we have already shown. We infer that $\bar p \leq 0$ in a similar way as above.
\end{proof}

\begin{Example}
Let $a \in L^{\infty}(\Omega)$ with $a \geq 0$ a.e. in $\Omega$. Then the functions $\phi(x,y)=e^{a(x)y}-1$ and $\phi(x,y)=a(x) | y|^{q-2}y \; (q \geq 3)$
are convex on $\mathbb{R}$ and $[0,\infty)$ respectively. Hence if $K=\emptyset$ and  $u_a,u_b$ and $y_0$ are chosen as in Lemma \ref{pleq0}, then Theorem \ref{sign} and Lemma \ref{pleq0} imply that
a solution of the necessary first order conditions will be the unique global minimum of $(\mathbb{P})$.
\end{Example}

\subsection{Conditions involving a bound on $\Vert \bar p \Vert_{L^q}$}

As mentioned above it will in general not be possible to establish a sign on the adjoint variable $\bar p$, so that one is left with trying to
bound $| R(u)|$ in terms of $\frac{1}{2} \| y-\bar y \|_{L^2(\Omega)}^2 + \frac{\alpha}{2} \| u -\bar u\|_{L^2(\Omega)} ^2$. In what follows
we shall assume that there exists $\gamma \in [0,1)$ and $M \geq 0$ such that

\begin{equation} \label{assumption: 1}
\Big| \frac{\phi_y(x,y_2) - \phi_y(x,y_1)}{y_2-y_1} \Big| \leq M \, \Bigl( \frac{\phi(x,y_2) - \phi(x,y_1)}{y_2-y_1} \Bigr)^{\gamma} 
\end{equation}

for almost all $x \in \Omega$ and  for all $y_1,y_2 \in \mathbb{R}, y_1 \neq y_2$. Note that (\ref{assumption: 1}) holds with $\gamma=0$ if $y \mapsto \phi_y(x,y)$ is globally
Lipschitz uniformly in  $x \in \Omega$. Furthermore, it is not difficult to verify that (\ref{assumption: 1}) is satisfied with $\gamma=\frac{1}{r}$ provided that
(\ref{assumption: 1a}) holds.

\begin{Example}
Let $\phi(x,y)=a(x)|y|^{q-2}y$, where $q \geq 3$ and $a \in L^\infty(\Omega)$ with $a(x) \geq 0$ a.e. in $\Omega$. Then, $\phi$ satisfies
(\ref{assumption: 1}) with $\gamma=\tfrac{q-3}{q-2}$ and $M=(q-2)(q-1)^{\frac{1}{q-2}} \|a\|^{\frac{1}{q-2}}_{L^\infty(\Omega)}$. \\[2mm]
\end{Example}

In what follows we shall make use of the elementary inequality (see e.g. \cite[Lemma 7.1]{ali2016global})
\begin{equation} \label{element}
a^{\lambda} b^\mu \leq \frac{\lambda^{\lambda} \mu^{\mu}}{(\lambda+\mu)^{\lambda+\mu}} (a+b)^{\lambda+\mu}, \quad a, b \geq 0,  \lambda, \mu >0,
\end{equation}
as well as of the Gagliardo--Nirenberg interpolation inequality
\begin{equation}  \label{gnineq}
\Vert f \Vert_{L^q} \leq C_q \Vert f \Vert_{L^2}^{1-\theta} \Vert \nabla f \Vert_{L^2}^{\theta}
\end{equation}
where $\theta= d(\frac{1}{2} - \frac{1}{q})$ and $2 \leq q < \infty$ if $d=2$ and $2 \leq q \leq 6$ if $d=3$. Explicit values for the constant $C_q$ in
(\ref{gnineq}) can e.g. be found in \cite{nasibov1990} and \cite{veling2002lower}, see also \cite[Theorem 7.3]{ali2016global}. \\
Before we state our main result we mention that it is well--known that $\bar p \in W^{1,s}_0(\Omega)$ for all $s \in [1,\frac{d}{d-1})$. In particular we infer with the help of a standard embedding result that 
\begin{equation} \label{p1} 
\bar p \in L^q(\Omega) 
\left\{
\begin{array}{ll}
\mbox{ for every } 1 \leq q < \infty & \mbox{ if } d=2;  \\
\mbox{ for every } 1 \leq q < 3 & \mbox{ if } d=3.
\end{array}
\right.  
\end{equation}
Furthermore, we have that
\begin{equation}  \label{p2}
\bar p \in L^{\infty}(\Omega) \mbox{ if } K = \emptyset \mbox{ or } K=\bar \Omega \mbox{ with }y_a,y_b \in W^{2,\infty}(\Omega). 
\end{equation}
In order to see (\ref{p2}) we note that $\bar p \in H^2(\Omega) \hookrightarrow L^{\infty}(\Omega)$ by elliptic regularity theory if $K= \emptyset$. On the
other hand, if $K=\bar \Omega$ with $y_a,y_b \in W^{2,\infty}(\Omega)$ we may apply Theorem 3.1 and Section 4.2 in \cite{casas2014new} to obtain that $\bar p \in L^{\infty}(\Omega)$. \\

\begin{Theorem} \label{main1}
Assume that $\phi$ satisfies (\ref{assumption: 1}) and let $(\bar u,\bar y,\bar p,\bar \mu) \in U_{ad} \times (H^2(\Omega) \cap H^1_0(\Omega)) \times L^2(\Omega) \times \mathcal M(K)  $ be
a solution of (\ref{oc:state})--(\ref{oc:VI measure}).
Furthermore, choose $q>1$ such that
\begin{equation} \label{qcond}
\displaystyle \frac{1}{1-\gamma}<q< \infty \; \mbox{ if } d=2; \quad 
\frac{3}{2(1-\gamma)} \leq q <3 \; \mbox{ if } d=3
\end{equation}
and define for $t:= \frac{2q(1-\gamma)}{q(1-\gamma)-1}$ and $\rho:=\frac{d}{2q} + \gamma$ the quantity
\begin{equation}  \label{etadef}
\displaystyle
\eta(\alpha,q,d):= \Bigl(\frac{1-\gamma}{2-\gamma} \Bigr)^{\gamma-1}  M^{-1}  C_{t}^{2(\gamma-1)} \alpha^{\frac{\rho}{2}} (\frac{d}{2q})^{-\frac{d}{2q}} \gamma^{-\gamma} (2-\rho)^{\frac{\rho}{2}-1} \rho^{\frac{\rho}{2}},
\end{equation}
where $C_t$ is the constant in (\ref{gnineq}). If the inequality
\begin{equation}
\label{18strich}
\|\bar p \|_{L^q} \leq \eta(\alpha,q,d)
\end{equation}
is satisfied, then $\bar u$ is a global minimum for Problem~$(\mathbb{P})$. If the inequality (\ref{18strich}) is strict, then
$\bar u$ is the unique global minimum. The assertions  hold for $\frac{3}{2(1-\gamma)} \leq q<\infty$ and $d=3$ provided that $K = \emptyset$ or $K=\bar \Omega$ with $y_a,y_b \in W^{2,\infty}(\Omega)$.
\end{Theorem}
\begin{proof} To begin, note that (\ref{p1}) and (\ref{p2}) imply that $\bar p \in L^q(\Omega)$ for the cases that we consider. Our starting point is again (\ref{eq2}) in
which we write the remainder term as
\begin{equation} \label{R_integral}
R(u)= \int\limits_\Omega \bar p(y-\bar y)\int\limits_0^1 [\phi_y(\cdot, \bar y+t(y-\bar y))-\phi_y(\cdot,\bar y)] dt \, dx. 
\end{equation}
We claim that for all $y_1,y_2 \in \mathbb{R}, y_1 \neq y_2$ we have
\begin{eqnarray}
\lefteqn{ \hspace{-1.5cm}
\Big| \int_0^1 [ \phi_y(\cdot,y_1+t(y_2-y_1))- \phi_y(\cdot,y_1)] dt \Big|}  \label{phiydif} \\
& \leq & L_\gamma | y_2 - y_1|^{1-2\gamma}  \bigl( (\phi(\cdot,y_2) - \phi(\cdot,y_1))(y_2-y_1) \bigr)^{\gamma}, \nonumber
\end{eqnarray}
where $L_\gamma= M \bigl( \frac{1-\gamma}{2-\gamma} \bigr)^{1-\gamma}$. To see this, let us suppress temporarily the dependence on $x$ and introduce
\begin{displaymath}
\phi_{\epsilon}(y):= \int_{\mathbb R} \zeta_{\epsilon}(z) \, \phi(y-z) \, dz, \quad y \in \mathbb{R},
\end{displaymath}
where $(\zeta_{\epsilon})_{0 < \epsilon <1} \subset C^{\infty}_0(\mathbb{R})$ is a sequence of mollifiers satisfying
\begin{displaymath}
\zeta_{\epsilon} \geq 0, \; \mbox{supp} \zeta_{\epsilon} \subset [-\epsilon,\epsilon], \mbox{ and } \int_{\mathbb{R}} \zeta_{\epsilon}(z) \, dz =1.
\end{displaymath}
Since $\phi_{\epsilon}'(y)= \int_{\mathbb{R}} \zeta_{\epsilon}(z) \phi'(y-z) \, dz$  we have that
\begin{displaymath}
\phi_{\epsilon}''(y)= \lim_{h \rightarrow 0} \int_{\mathbb{R}} \zeta_{\epsilon}(z) \frac{\phi'(y+h-z)-\phi'(y-z)}{h} \, dz
\end{displaymath}
so that we obtain with the help of  (\ref{assumption: 1})  and H\"older's inequality
\begin{eqnarray*}
| \phi_{\epsilon}''(y) | & \leq & M \int_{\mathbb{R}} \zeta_{\epsilon}(z) ( \phi'(y-z) )^{\gamma} \, dz = M \int_{\mathbb{R}} (\zeta_{\epsilon}(z))^{1-\gamma} \bigl( \zeta_{\epsilon}(z) \,  \phi'(y-z) \bigr)^{\gamma} \, dz \\
& \leq & M \Bigl( \int_{\mathbb{R}} \zeta_{\epsilon}(z) \phi'(y-z) \, dz \Bigr)^{\gamma} = M (\phi_{\epsilon}'(y))^{\gamma}.
\end{eqnarray*}
We may therefore apply Lemma 7.2 in \cite{ali2016global} for $\gamma \in (0,1)$ to deduce that
\begin{displaymath}
\Big| \int_0^1 [ \phi_{\epsilon}'(y_1+t(y_2-y_1))- \phi_{\epsilon}'(y_1)] dt \Big|  \leq L_\gamma | y_2 - y_1|  \Bigl( \int_0^1 \phi_{\epsilon}'(y_1 + t(y_2 - y_1)) \, dt \Bigr)^{\gamma},
\end{displaymath}
but the above estimate easily extends to the case $\gamma=0$. The bound (\ref{phiydif}) now follows by sending $\epsilon \rightarrow 0$. 
If we insert (\ref{phiydif}) into (\ref{R_integral}) we find that
\begin{eqnarray}
|  R(u)  | &\leq & L_\gamma \,  \int_\Omega |\bar p | \,  |y-\bar y|^{2-2\gamma}   \bigl( (\phi(\cdot,y)-\phi(\cdot,\bar y))(y-\bar y) \bigr)^{\gamma}  \, dx \label{intermediate}  \\
& \leq & L_\gamma \,   \|\bar p\|_{L^q}  \| y- \bar y\|^{2(1-\gamma)}_{L^{2s(1-\gamma)}} \Bigl( \int_{\Omega}  (\phi(\cdot,y)-\phi(\cdot,\bar y))(y-\bar y) dx \Bigr)^{\gamma}, \nonumber
\end{eqnarray}
where we have used H\"older's inequality with exponents $q, r=\frac{1}{\gamma}$ and $s=\frac{q}{q(1-\gamma)-1}$. Note that 
\begin{displaymath}
2s(1-\gamma) = \frac{2q(1-\gamma)}{q(1-\gamma)-1} = t \in
\left\{
\begin{array}{ll}
(2, \infty), & \mbox{ if } d=2; \\
(2,6], & \mbox{ if } d=3
\end{array}
\right.
\end{displaymath}
in view of our assumptions on $q$. We may therefore use  (\ref{gnineq}) in order to estimate $\Vert y - \bar y \Vert_{L^t}$ and obtain with 
\begin{displaymath}
\theta = d \bigl( \frac{1}{2} - \frac{1}{t} \bigr) =  \frac{d}{2q(1-\gamma)} \quad \mbox{ and hence } \quad 2(1-\gamma) \theta = \frac{d}{q}
\end{displaymath}
that 
\begin{eqnarray*}
\lefteqn{ | R(u) | }  \\ 
& \leq &  L_\gamma \, C_{t}^{2(1-\gamma)}  \|\bar p\|_{L^q}    \| y- \bar y\|^{2(1-\gamma)- \frac{d}{q}}_{L^2}  \| \nabla( y- \bar y) \|^{\frac{d}{q}}_{L^2} \Bigl( \int_{\Omega}  (\phi(\cdot,y)-\phi(\cdot,\bar y))(y-\bar y) dx \Bigr)^{\gamma}.
\end{eqnarray*}
Applying  (\ref{element}) 
with $\lambda=\frac{d}{2q}$ and $\mu=\gamma$ and recalling that $\rho=\frac{d}{2q}+\gamma$ we may continue 
\begin{eqnarray*}
| R(u) | & \leq &   L_\gamma \, C_{t}^{2(1-\gamma)}  \|\bar p\|_{L^q}  \| y- \bar y\|^{2(1-\gamma) -\frac{d}{q}}_{L^2} \\[2mm] 
& & \times \frac{(\frac{d}{2q})^{\frac{d}{2q}} \gamma^{\gamma}}{\rho^\rho}  \Bigl( \Vert \nabla( y - \bar y) \Vert_{L^2}^2 +
\int_{\Omega}  (\phi(\cdot,y)-\phi(\cdot,\bar y))(y-\bar y) dx \Bigr)^\rho.
\end{eqnarray*}
If we take the difference of the PDEs satisfied by $\bar y$ and $y$ and test it with $y-\bar y$ we easily deduce that
\begin{displaymath}  
\Vert \nabla(y - \bar y) \Vert_{L^2}^2 + \int_{\Omega} \bigl( \phi(\cdot,y)- \phi(\cdot,\bar y) \bigr) (y-\bar y) \, dx  \leq \Vert y - \bar y \Vert_{L^2} \Vert u - \bar u \Vert_{L^2},
\end{displaymath}
which yields
\begin{eqnarray*}
| R(u)| 
& \leq & L_\gamma \,  C_{t}^{2(1-\gamma)} \frac{(\frac{d}{2q})^{\frac{d}{2q}} \gamma^{\gamma}}{\rho^\rho} \|\bar p\|_{L^q}  \| y- \bar y\|^{2(1-\gamma) -\frac{d}{q}+ \rho }_{L^2} \Vert u - \bar u \Vert_{L^2}^\rho \\
& = & 2 L_\gamma \,  C_{t}^{2(1-\gamma)} \alpha^{-\frac{\rho}{2}} \frac{(\frac{d}{2q})^{\frac{d}{2q}} \gamma^{\gamma}}{\rho^\rho} \|\bar p\|_{L^q} \bigl( \frac{1}{2} \| y- \bar y\|_{L^2}^2 \bigr)^{1- \frac{\rho}{2} }
\bigl( \frac{\alpha}{2} \Vert u - \bar u \Vert^2_{L^2} \bigr)^{\frac{\rho}{2}}.
\end{eqnarray*}
Using once more (\ref{element}), this time with $\lambda=1-\frac{\rho}{2}, \mu=\frac{\rho}{2}$ we finally deduce that
\begin{eqnarray*}
| R(u) | & \leq & 2 L_\gamma \, C_{t}^{2(1-\gamma)} \alpha^{-\frac{\rho}{2}} \frac{(\frac{d}{2q})^{\frac{d}{2q}} \gamma^{\gamma}}{\rho^\rho}  \bigl(1-\frac{\rho}{2} \bigr)^{1-\frac{\rho}{2}} \bigl(\frac{\rho}{2} \bigr)^{
\frac{\rho}{2}} \|\bar p\|_{L^q} \bigl( \frac{1}{2} \| y- \bar y\|_{L^2}^2 + \frac{\alpha}{2} \Vert u - \bar u \Vert^2_{L^2} \bigr) \\
& = & L_\gamma \,  C_{t}^{2(1-\gamma)} \alpha^{-\frac{\rho}{2}} (\frac{d}{2q})^{\frac{d}{2q}} \gamma^{\gamma} (2-\rho)^{1-\frac{\rho}{2}} \rho^{-\frac{\rho}{2}} \|\bar p\|_{L^q} \bigl( \frac{1}{2} \| y- \bar y\|_{L^2}^2 + \frac{\alpha}{2} \Vert u - \bar u \Vert^2_{L^2} \bigr).
\end{eqnarray*}
If we use this estimate in (\ref{eq2}) and recall (\ref{etadef}) as well as $L_\gamma=M \bigl( \frac{1-\gamma}{2-\gamma} \bigr)^{1-\gamma}$  we infer that $J(u)-J(\bar u) \ge 0$ provided that \eqref{18strich} holds, so that $\bar u$ is a global solution of problem $(\mathbb{P})$. If the inequality in \eqref{18strich} is strict, then $\bar u$ is the unique global minimum of problem $(\mathbb{P})$.
\end{proof}

\begin{Remark}
Suppose that $d=2$ and that $\phi$ satisfies (\ref{assumption: 1a}) for some $r>1, M \geq 0$, so that (\ref{assumption: 1}) holds with $\gamma=\frac{1}{r}$. If we set $q:=\frac{3r-2}{r-1}$, then $q$ satisfies (\ref{qcond})
while $t=q$ and $\rho= \frac{1}{q}+\frac{1}{r}= \frac{r+q}{rq}$, so that Theorem \ref{main0} is a special case of Theorem \ref{main1}. 
\end{Remark}


\setcounter{equation}{0}

\section{Variational discretization}
\label{sec:variational}

In this section we consider the case $d=2$ and let
$\mathcal T_h$ be an admissible triangulation of  $\Omega \subset \mathbb{R}^2$. We 
introduce the following spaces of linear finite elements:
\begin{align*}
X_{h} &:= \{ v_h \in C^0(\bar\Omega) : v_{h|T} \mbox{ is a linear polynomial on each }  T \in \mathcal{T}_h\}, \\
X_{h0} &:= \{ v_h \in X_h : {v_h}_{ | \partial\Omega}=0   \}.
\end{align*}
The Lagrange interpolation operator $I_h$ is defined by
\begin{equation*}
I_h:C^0(\bar\Omega) \to X_h, \qquad I_h y:= \sum_{i=1}^n y(x_i) \phi_i,
\end{equation*}
where $x_1,\ldots,x_n$ denote the nodes in the triangulation  $\mathcal{T}_h$ and $\{\phi_1, \ldots, \phi_n\}$  is the set of basis functions of the space $X_h$ which satisfy $\phi_i(x_j)=\delta_{ij}$.
We discretize (\ref{semilinear}), (\ref{bc}) using numerical integration for the nonlinear part: for a given $u \in L^2(\Omega)$, find $y_h \in X_{h0}$ such that
\begin{equation}
\label{WPDE:semilinear state h}
\int_\Omega \nabla y_h   \cdot  \nabla v_h + I_h[\phi(\cdot,y_h) v_h] \, dx = \int_\Omega u v_h \,dx \quad \forall \, v_h \in X_{h0}.
\end{equation}
Using the monotonicity of $y \mapsto \phi(\cdot,y)$ and the Brouwer fixed-point theorem one can show that (\ref{WPDE:semilinear state h}) admits a unique solution  $y_h=:\mathcal G_h(u) \in X_{h0}$.
The variational discretization (see \cite{hinze2005variational}) of Problem~$(\mathbb{P})$ then reads:
\begin{displaymath}
(\mathbb{P}_h) \quad
\begin{array}{rcl}
&&  \min_{u \in U_{ad}} J_h(u):=\frac{1}{2} \Vert y_h-y_0 \Vert_{L^2(\Omega)}^2  + \dfrac{\alpha}{2} \Vert u \Vert_{L^2(\Omega)} ^2 \\
&& \mbox{ subject to } y_h = \mathcal G_h(u), \, y_a(x_j) \leq y_h(x_j) \leq y_b(x_j), \; x_j \in \mathcal N_h,
\end{array}
\end{displaymath}
where $\mathcal N_h := \lbrace x_j \, | \, x_j \mbox{ is a node of } T \in \mathcal T_h, \mbox{ such that } T \cap K \neq \emptyset \rbrace$. It can be shown that $(\mathbb{P}_h)$ has a  solution, provided that a feasible point exists. In practice, candidates for solutions are
calculated by solving the system of necessary first order conditions which reads: find $\bar u_h \in U_{ad}, \bar y_h \in X_{h0}, \bar{p}_h \in X_{h0}, \bar{\mu}_j \in \mathbb{R}, x_j \in \mathcal N_h$ such that
$y_a(x_j) \leq y_h(x_j) \leq y_b(x_j), x_j \in \mathcal N_h$ and
\begin{align}
&\int_{\Omega} \nabla \bar y_h \cdot \nabla v_h + I_h[\phi(\cdot,\bar y_h) v_h] \, dx  = \int_{\Omega} \bar u_h v_h \, dx  \quad \forall \, v_h \in X_{h0},   \label{oc:state h} \\
&\int_\Omega \nabla \bar p_h \cdot \nabla v_h + I_h[\phi_y(\cdot,\bar y_h)\bar p_h v_h] \, dx
 =\int_\Omega (\bar y_h -y_0) v_h \, dx  + \sum_{x_j \in \mathcal N_h} \bar \mu_j v_h(x_j)  \; \forall \, v_h \in X_{h0}, \label{oc:adjoint h}
 \\
& \int_\Omega (\bar p_h + \alpha \bar u_h)(u-\bar u_h) \, dx  \geq  0 \qquad \forall \, u \in U_{ad}, \label{oc:VI control h} \\
& \sum_{x_j \in \mathcal N_h} \bar \mu_j (y_j-\bar y_h(x_j))  \leq  0 \qquad \forall \, (y_j)_{x_j \in \mathcal N_h}, y_a(x_j) \leq y_j \leq y_b(x_j), x_j \in \mathcal N_h. \label{oc:VI measure h}
\end{align}

In order to formulate the analogue of Theorem \ref{main1} we introduce the following $h$--dependent norm on $X_h$:
\begin{displaymath}
\Vert v_h \Vert_{h,q}:= \bigl( \int_{\Omega} I_h[ | v_h |^q] dx \bigr)^{\frac{1}{q}}, \quad v_h \in X_h, \;  1 \leq q < \infty.
\end{displaymath}

\begin{Theorem}
\label{Thm:global minima h}
Suppose that $\phi$ and $q>1$ satisfy the conditions (\ref{assumption: 1}) and (\ref{qcond}) respectively and let
$\bar u_h \in U_{ad}$, $\bar y_h \in X_{h0}$, $\bar p_h \in X_{h0}$, $(\bar \mu_j)_{x_j \in \mathcal N_h}$ be a solution of (\ref{oc:state h})--(\ref{oc:VI measure h}).
If
\begin{equation}
\label{18strich h}
\|\bar p_h\|_{h,q} \leq (\frac{1}{4})^{1-\gamma-\frac{1}{q}} \, \eta(\alpha,q,2),
\end{equation}
then $\bar u_h$ is a global minimum for Problem~$(\mathbb{P}_h)$. If the inequality (\ref{18strich h}) is strict, then $\bar u_h$ is the unique global minimum.
\end{Theorem}
\begin{proof} Just as in the continuous case we obtain for $u \in U_{ad}$ with $y_h= \mathcal G_h(u)$
\begin{equation} \label{eq3}
J_h(u)-J_h(\bar u_h)   \geq   \frac{1}{2} \| y_h-\bar y_h \|_{L^2(\Omega)}^2 + \dfrac{\alpha}{2} \| u -\bar u_h\|_{L^2(\Omega)} ^2  - R_h(u), 
\end{equation}
where  
\begin{eqnarray} 
 R_h(u) & = & \int_{\Omega} I_h \left[\bigl( \phi(\cdot,y_h) - \phi(\cdot,\bar y_h) - \phi_y(\cdot,\bar y_h)(y_h - \bar y_h) \bigr) \bar p_h \right] \, dx \nonumber  \\
 & = &  \int_\Omega I_h \left[ \bar p_h (y_h -\bar y_h) \int\limits_0^1 (\phi_y(\cdot, \bar y_h+t(y_h-\bar y_h))-\phi_y(\cdot,\bar y_h))  dt \right] dx  \label{remainderh}.
\end{eqnarray}
If we use (\ref{phiydif}) then we obtain as above with the help of H\"older's inequality 
\begin{eqnarray*}
|  R_h(u)  | &\leq & M 2^{\gamma-1}  \int_\Omega I_h \left[ |\bar p_h | \,  |y_h -\bar y_h|^{2-2\gamma}   \bigl( (\phi(\cdot,y_h)-\phi(\cdot,\bar y_h))(y_h-\bar y_h) \bigr)^{\gamma} \right]  \, dx \\
& \leq & M 2^{\gamma-1}  \|\bar p_h\|_{h,q}  \| y_h- \bar y_h \|^{2(1-\gamma)}_{h,2s(1-\gamma)} \Bigl( \int_{\Omega}  I_h \left[(\phi(\cdot,y_h)-\phi(\cdot,\bar y_h))(y_h-\bar y_h) \right] dx \Bigr)^{\gamma},
\end{eqnarray*}
where $s=\frac{q}{q(1-\gamma)-1}$. Applying Lemma \ref{intest} in the Appendix we derive
\begin{displaymath}
| R_h(u) |  \leq  M 2^{\gamma-1} 4^{\frac{1}{s}}  \|\bar p_h\|_{h,q}  \| y_h- \bar y_h \|^{2(1-\gamma)}_{L^{2s(1-\gamma)}} \Bigl( \int_{\Omega}  I_h \left[(\phi(\cdot,y_h)-\phi(\cdot,\bar y_h))(y_h-\bar y_h) \right] dx \Bigr)^{\gamma},
\end{displaymath}
which is the analogue of (\ref{intermediate}). The rest of the proof now follows in the same way as in Theorem \ref{main1}, where we use  (\ref{WPDE:semilinear state h}) instead of the PDEs.
\end{proof}
We shall investigate condition (\ref{18strich h}) for different choices of $\phi$ and $q$ in the numerics section. From the numerical analysis point of view it is also possible to examine the convergence of
a sequence of solutions $(\bar u_h,\bar y_h,\bar p_h,(\bar \mu_j)_{x_j \in \mathcal N_h})_{0 < h < h_0}$ of (\ref{oc:state h})--(\ref{oc:VI measure h}) that satisfy (\ref{18strich h}) uniformly in $h$. Based on Theorem \ref{main0}, convergence in
$L^2(\Omega)$ of  $(\bar u_h)_{0<h< h_0}$ to a solution $\bar u$  of $(\mathbb{P})$ has been obtained in \cite[Theorem 4.2]{ali2016global}, while an error estimate is proved in \cite{ahmadoptimal,ADH2018}. We expect that
these results carry over to the generalized framework considered in this paper. In this context we also refer to \cite{neitzel2015finite} as a further contribution to the error analysis for optimal control of semilinear equations 
with pointwise bounds on the state. Contrary to our approach this work is based on second order sufficient optimality conditions for a local solution of the control problem and requires in particular a $C^2$--nonlinearity $\phi$.

\setcounter{equation}{0}

\section{Numerical experiments}\label{numerics}
In this section we conduct several numerical experiments related to Theorem~\ref{Thm:global minima h}. We consider $(\mathbb{P})$ with different choices for the nonlinearity $\phi$. For each choice we fix $\Omega:=(0,1)\times (0,1)$, while 
for the desired state $y_0$ we consider the following two scenarios:\\

\noindent
\textbf{A1:} (Reachable desired state) $y_0(x):=2\sin(2\pi x_1)\sin(2\pi x_2)$.

\noindent
\textbf{A2:} (Not reachable desired state) $y_0(x):=60+160(x_1(x_1-1)+x_2(x_2-1))$.\\

\noindent
For the control and state bounds we consider these three cases:\\

\noindent
\textbf{Case 1:} (Unconstrained problem) $u_b =-u_a= \infty$, $K=\emptyset$.

\noindent
\textbf{Case 2:} (Control constrained problem) $u_b =-u_a= 5$, $K=\emptyset$.

\noindent
\textbf{Case 3:} (State constrained problem) $u_b =-u_a= \infty$, $K=\bar \Omega, y_b \equiv-y_a \equiv 1$.\\

For $\alpha$ we report numerical results for the values $\alpha=10^i$, $i=-6,-5, \ldots, 3$.  The domain $\Omega$ is partitioned using a uniform triangulation with mesh size $h=2^{-5}\sqrt{2}$, and the discrete counterpart of the problem is as in Section~\ref{sec:variational}. The resulting discrete optimality system (\ref{oc:state h})--(\ref{oc:VI measure h}) is solved using the semismooth Newton method.

\begin{Example}
We consider $\phi(y):=y|y|$. Then, $\gamma=0$ with $M=2$. Taking $q=2$, the condition reads
\[
\|\bar p_h\|_{h,2} \leq \frac{1}{2} \, \eta(\alpha,2,2)
\] 
\end{Example}
with
\begin{align*}
C_4^{-2} \approx 2.381297723376159.
\end{align*}
The results are reported in Figure~\ref{F1}. We see that in the light of Theorem~\ref{Thm:global minima h}, the unique global solution of the considered control problem has been computed for all given values of $\alpha$, except for case~2 when $\alpha \leq 10^{-3}$. There, no conclusion can be derived. However, with the coefficient $a(x):=\frac{1}{8}$ we obtain a global unique solution for the whole considered parameter range, see Fig.~\ref{F4}.

\begin{Example}
	We consider $\phi(y):=y^3$. Then, $\gamma=0.5$ with $M=2\sqrt{3}$. Taking $q=3$, the condition reads
	\[
	\|\bar p_h\|_{L^3(\Omega)} \leq   \eta(\alpha,3,2)
	\] 
\end{Example}
with
\begin{align*}
C_6^{-1} \approx 1.616080082127768.
\end{align*}
The choice of $q=3$ is motivated by fact that among the possible choices of the  Gagliardo-Nirenberg constant the value of $C_6$ is among the smallest possible ones, see \cite[Figure 4]{ali2016global}. The integrals involving $\phi$, and the norm $\|\bar p_h\|_{L^3(\Omega)}$ are computed exactly. The results are reported in Figure~\ref{F2}. We for comparison also include the results for $q=4$ which correspond to the findings of \cite[Example 2]{ali2016global}. As one can see this choice in some situations delivers larger uniqueness intervalls for $\alpha$. Overall, uniqueness of the global solution can be deduced for certain ranges of the parameter $\alpha$, where it is more likely in the case of a reachable desired state $y_0$.

\begin{Example}
	We consider $\phi(y):=y^5$. Then, $\gamma=3/4$ with $M=4\times (5)^{1/4}$. Taking $q=6$, the condition reads
	\[
	\|\bar p_h\|_{L^6(\Omega)} \leq   \eta(\alpha,6,2)
	\] 
\end{Example}
with
\begin{align*}
C_6^{-1/2} \approx 1.271251384316953.
\end{align*}
The choice of $q=6$ is motivated as in the previous example. This then is the situation of \cite[Example 3]{ali2016global}. For comparison we also include the results obtained with quadrature based on the estimate \eqref{18strich h}. As one can see the differences in both approaches (exact integration versus quadrature) is negligible. The results are reported in Figure~\ref{F3}.

\begin{figure}[p]
        \centering
        \begin{subfigure}[h!]{0.5\textwidth}
                \includegraphics[trim = 40mm 80mm 30mm 70mm, clip, width=\textwidth]{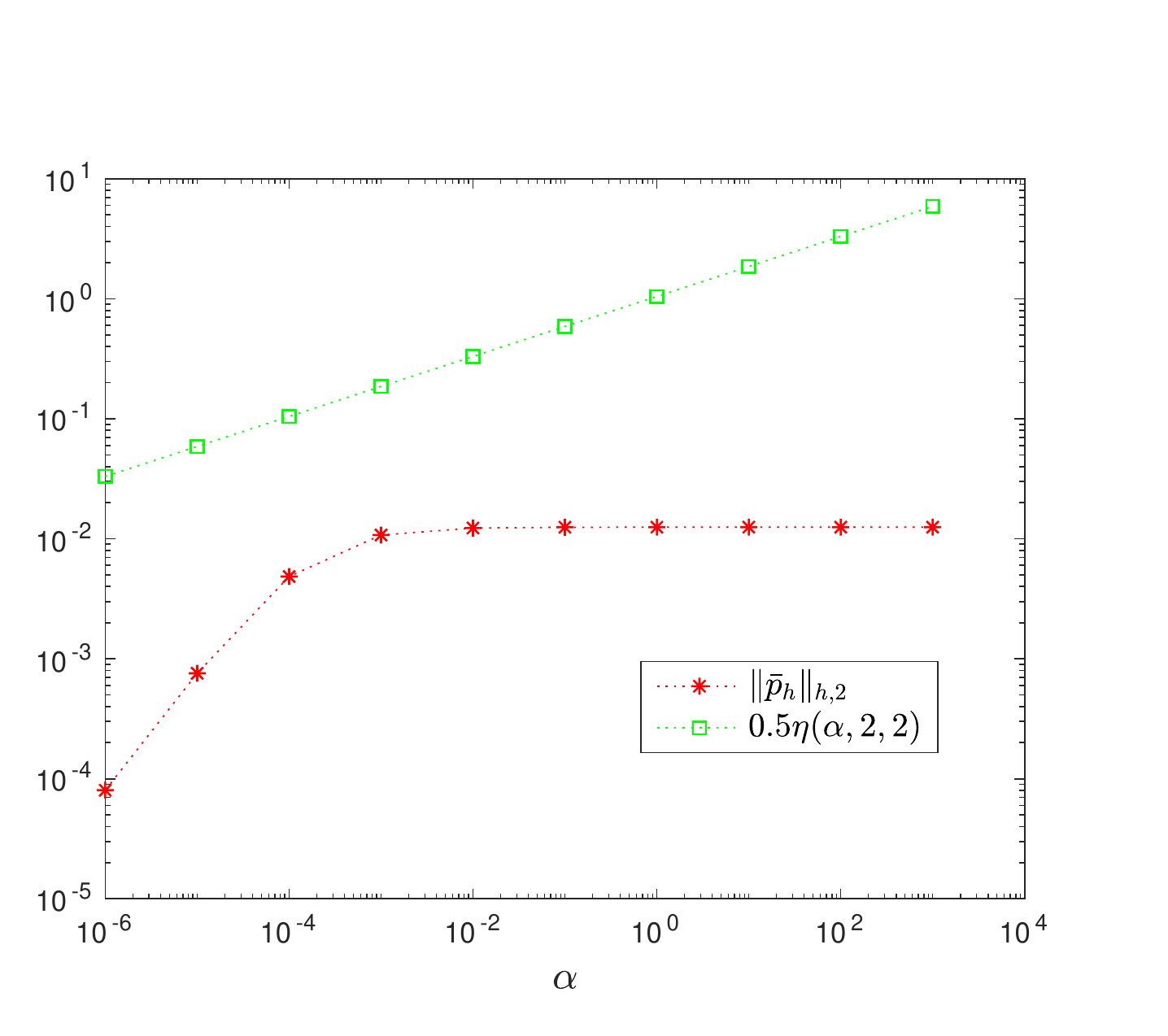}
                \caption{Case 1 with A1}
        \end{subfigure}\hfill%
         \begin{subfigure}[h!]{0.5\textwidth}
                \includegraphics[trim = 40mm 80mm 30mm 70mm, clip, width=\textwidth]{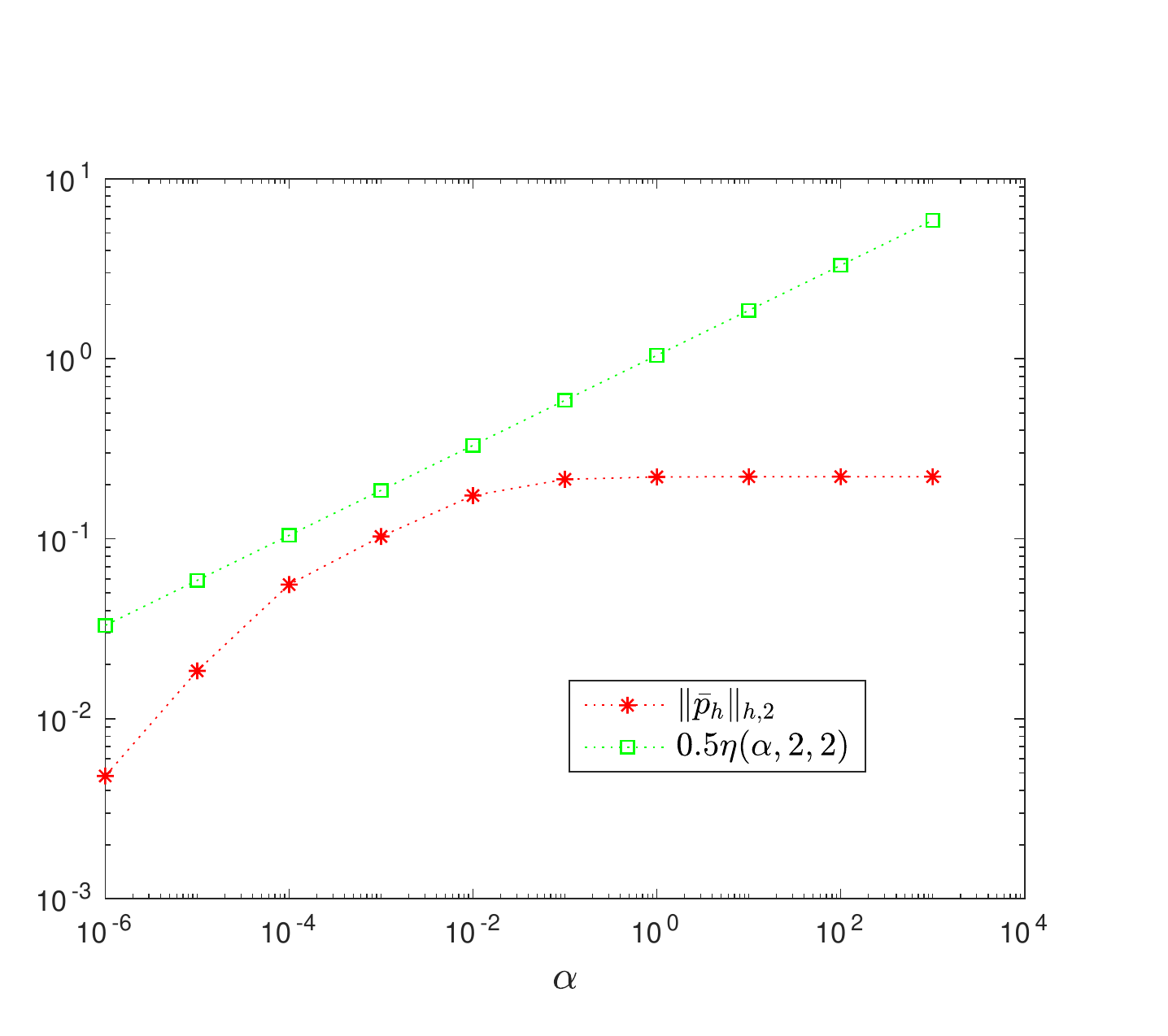}
                \caption{Case 1 with A2}
        \end{subfigure}\\%
        \begin{subfigure}[h!]{0.5\textwidth}
                \includegraphics[trim = 40mm 80mm 30mm 70mm, clip, width=\textwidth]{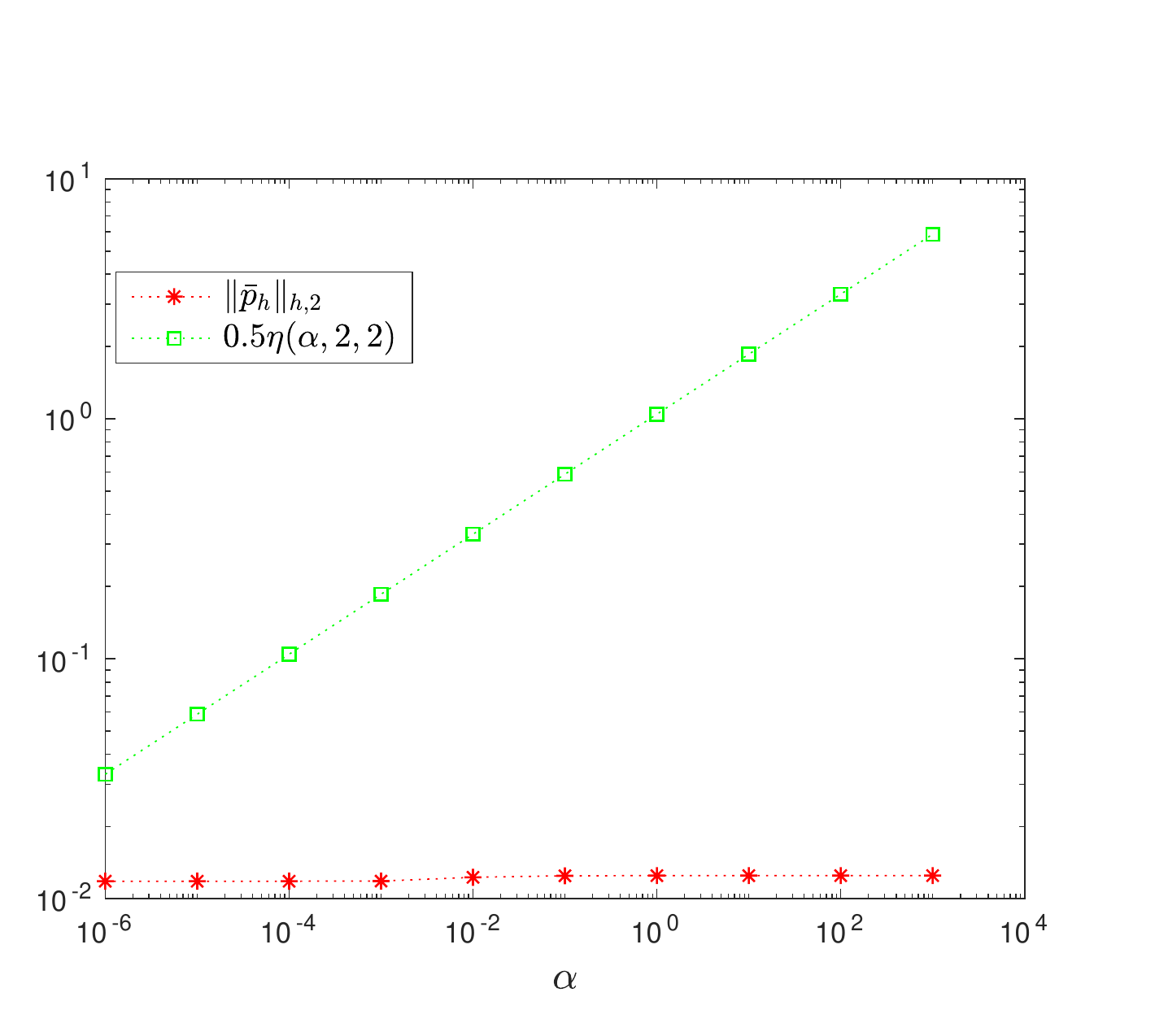}
                \caption{Case 2 with A1}
        \end{subfigure}\hfill%
        \begin{subfigure}[h!]{0.5\textwidth}
                \includegraphics[trim = 40mm 80mm 30mm 70mm, clip, width=\textwidth]{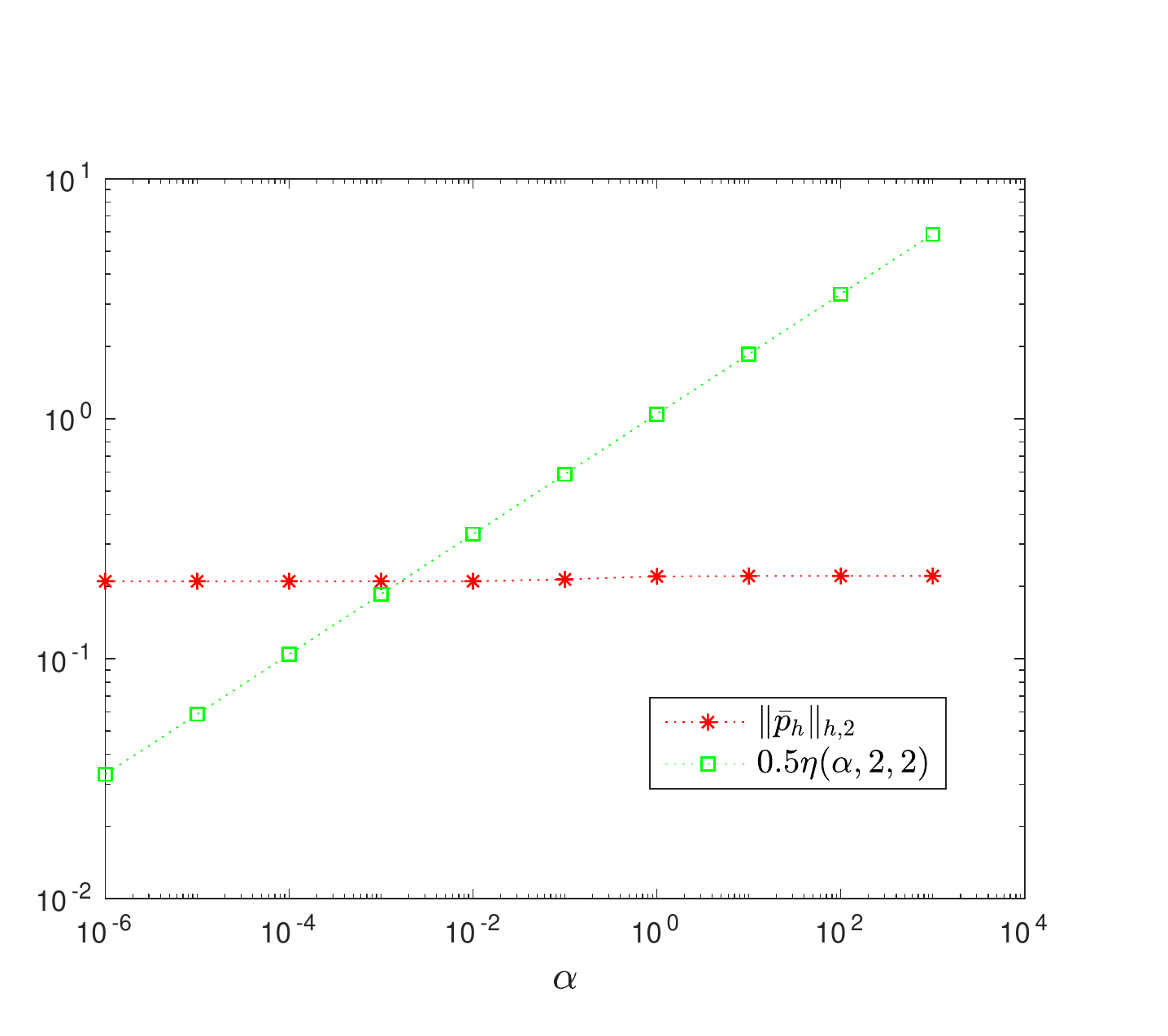}
                \caption{Case 2 with A2}
        \end{subfigure}\\%
        \begin{subfigure}[h!]{0.5\textwidth}
                \includegraphics[trim = 40mm 80mm 30mm 70mm, clip, width=\textwidth]{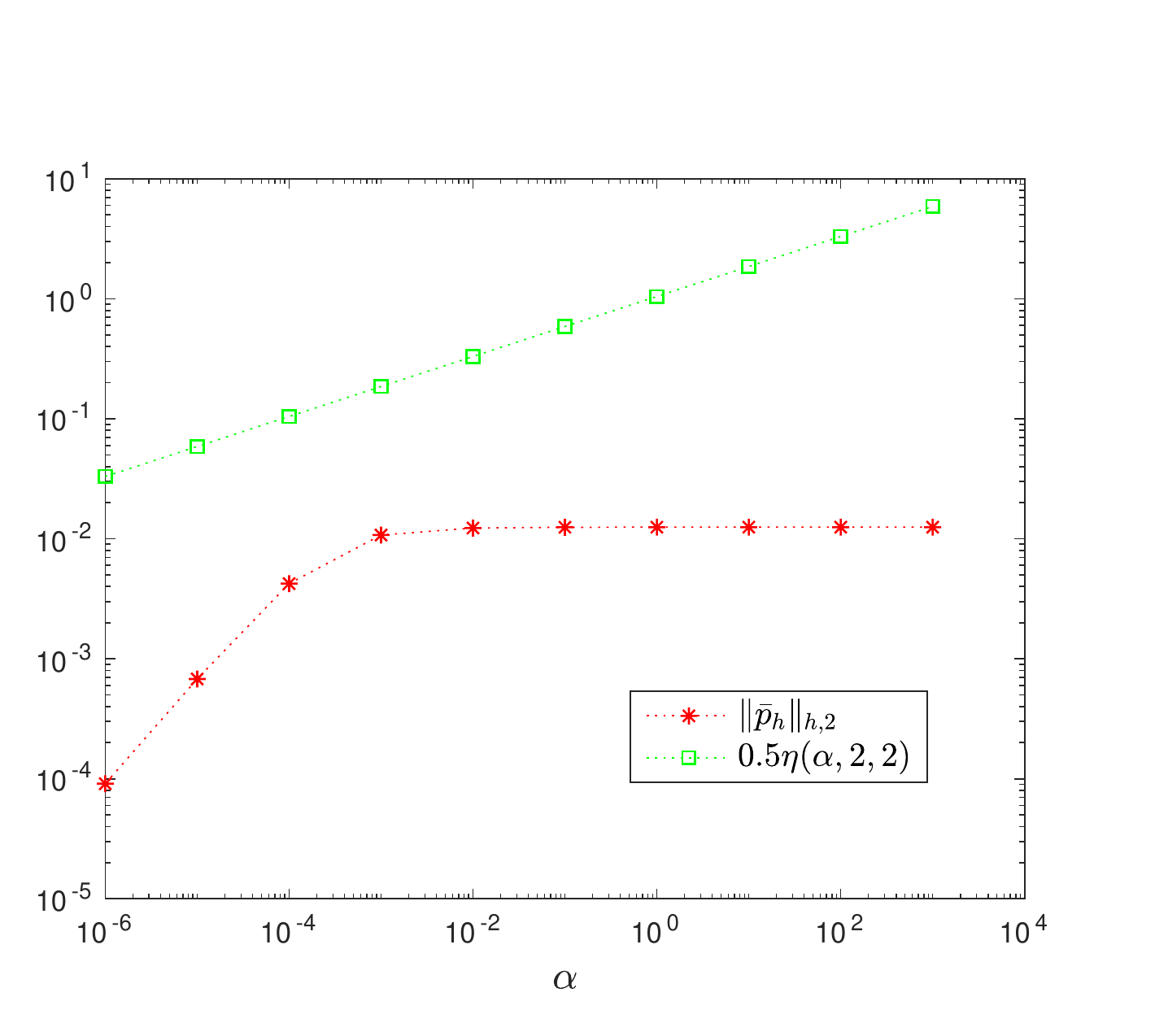}
                \caption{Case 3 with A1}
        \end{subfigure}\hfill%
        \begin{subfigure}[h!]{0.5\textwidth}
                \includegraphics[trim = 40mm 80mm 30mm 70mm, clip, width=\textwidth]{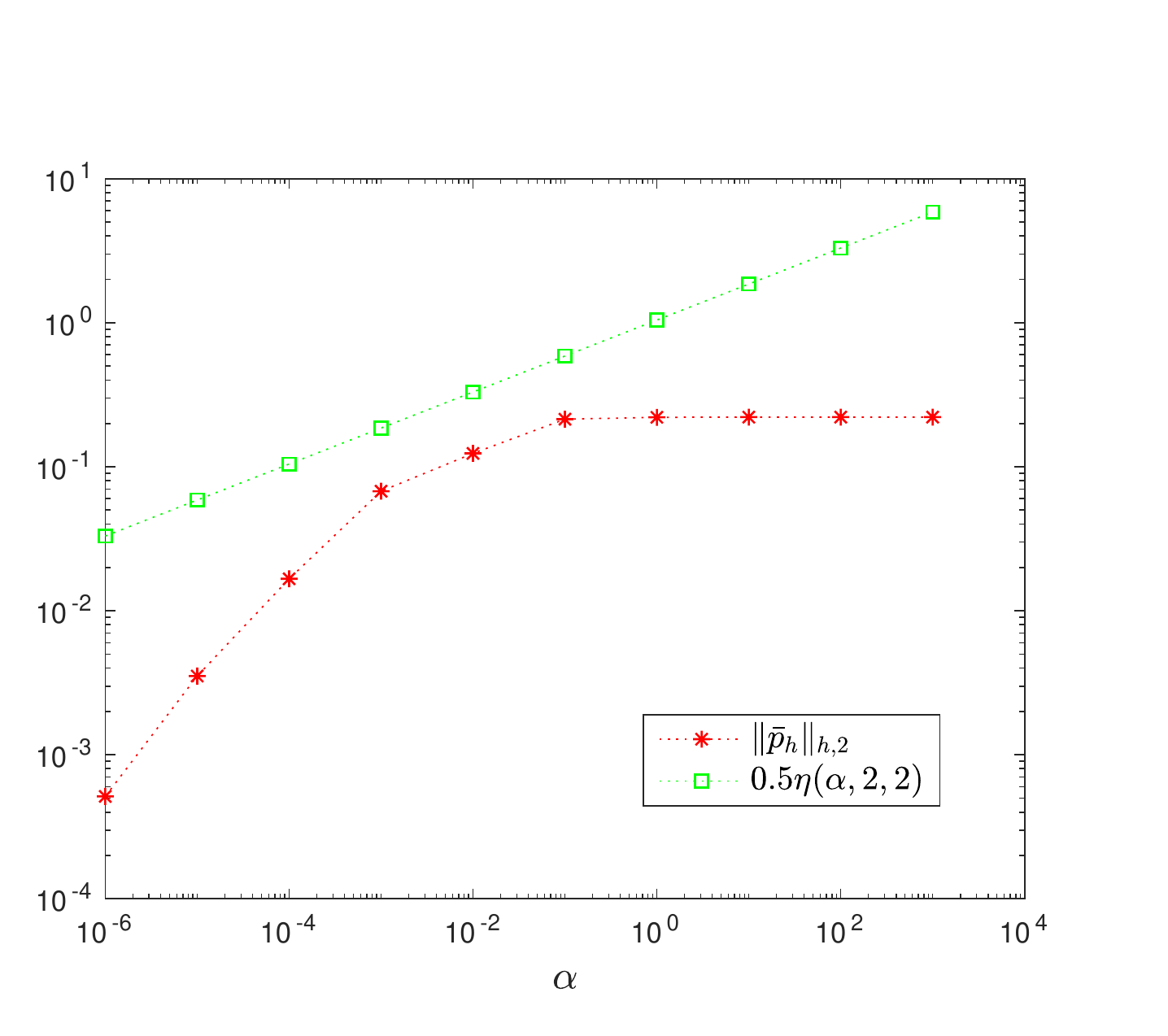}
                \caption{Case 3 with A2}
        \end{subfigure}%
        \caption{Results for $\phi(s)=s|s|$}
        \label{F1}
\end{figure}

\begin{figure}[htb]
        \centering
         \includegraphics[width=.4\textwidth]{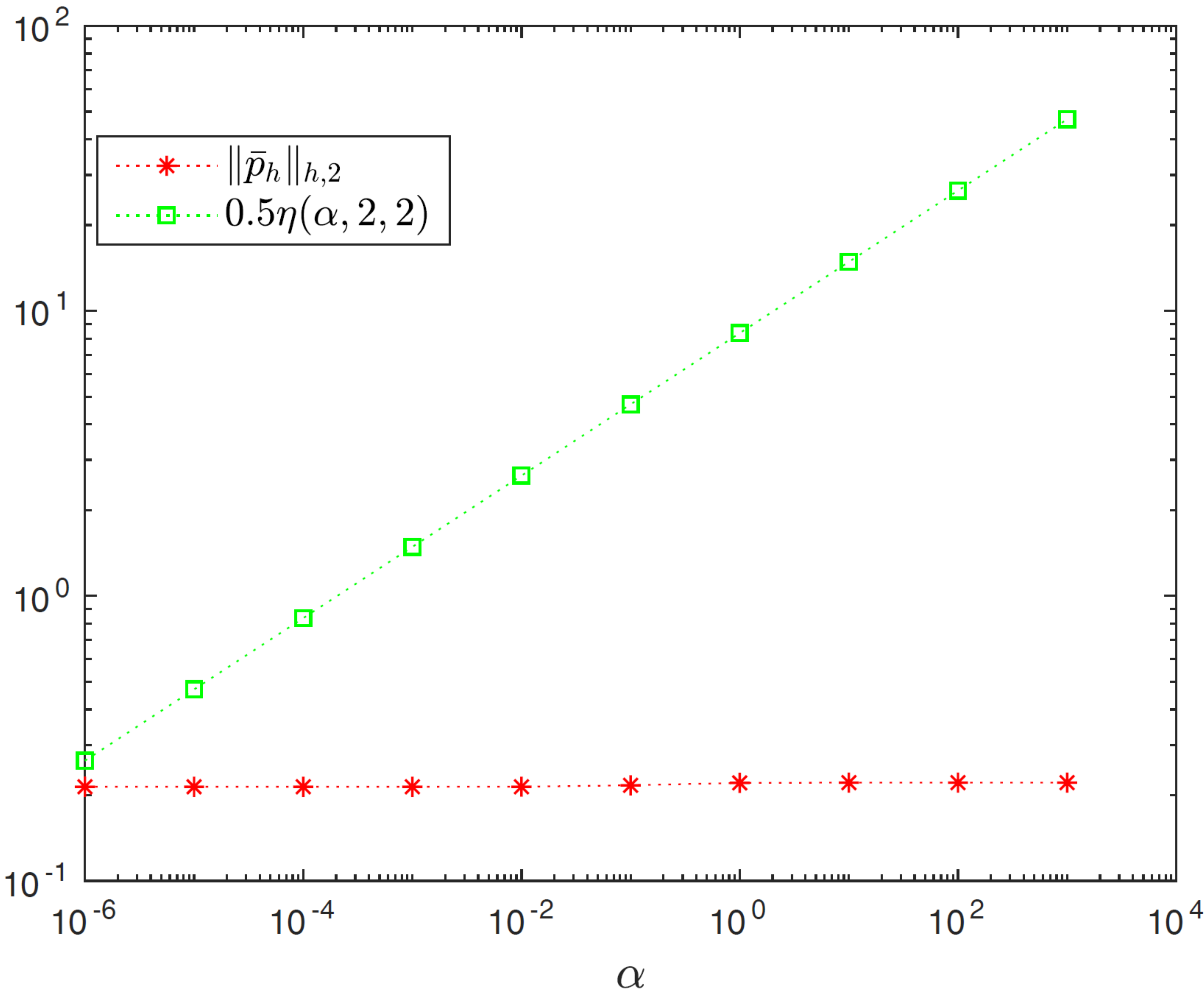}
        \caption{Case 2 with A2 for $\phi(s)=\frac{1}{8}s|s|$} \label{F4}
\end{figure}

\begin{figure}[p]
        \centering
        \begin{subfigure}[h!]{0.5\textwidth}
                \includegraphics[trim = 40mm 80mm 30mm 70mm, clip, width=\textwidth]{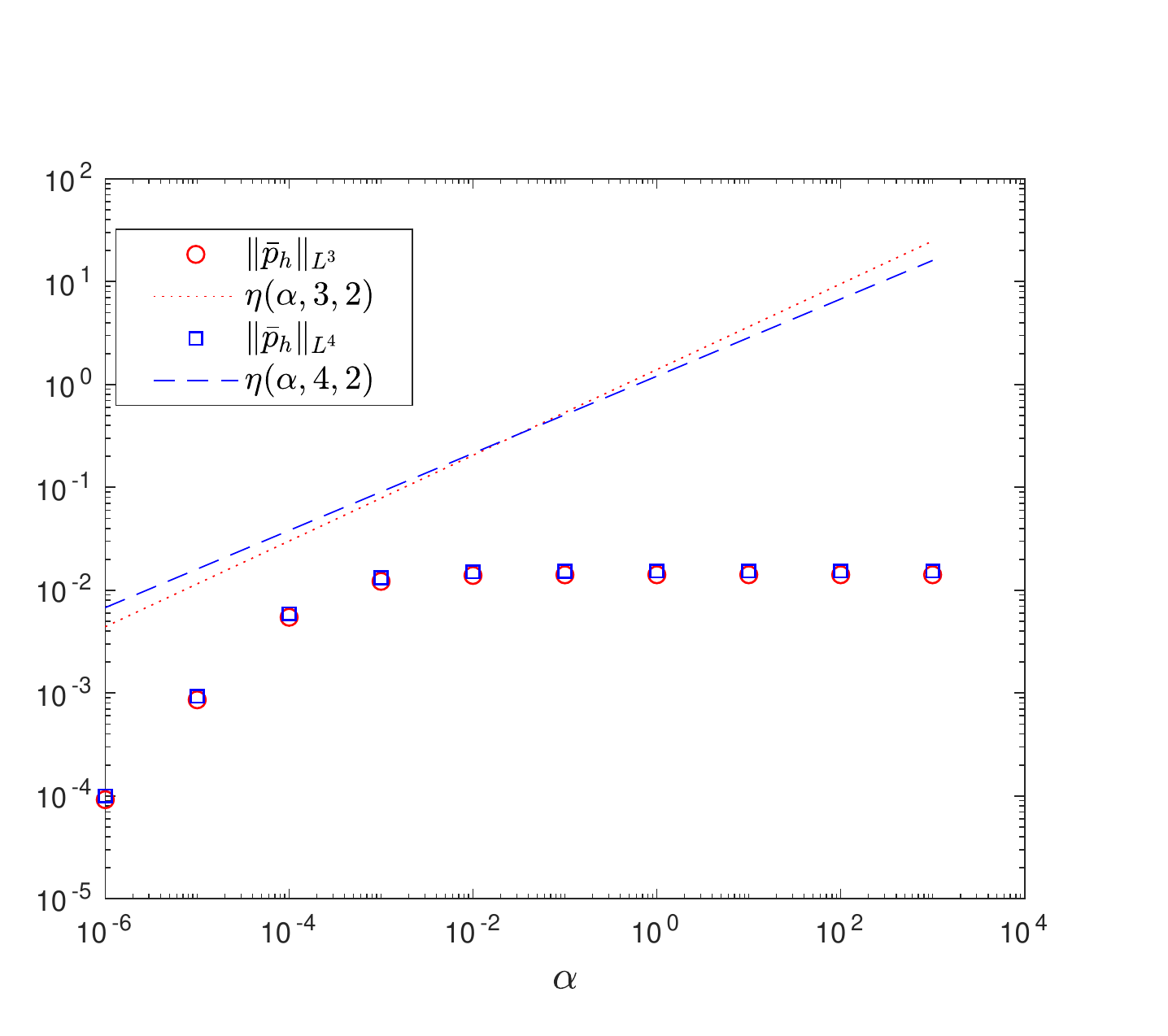}
                \caption{Case 1 with A1}
        \end{subfigure}\hfill%
         \begin{subfigure}[h!]{0.5\textwidth}
                \includegraphics[trim = 40mm 80mm 30mm 70mm, clip, width=\textwidth]{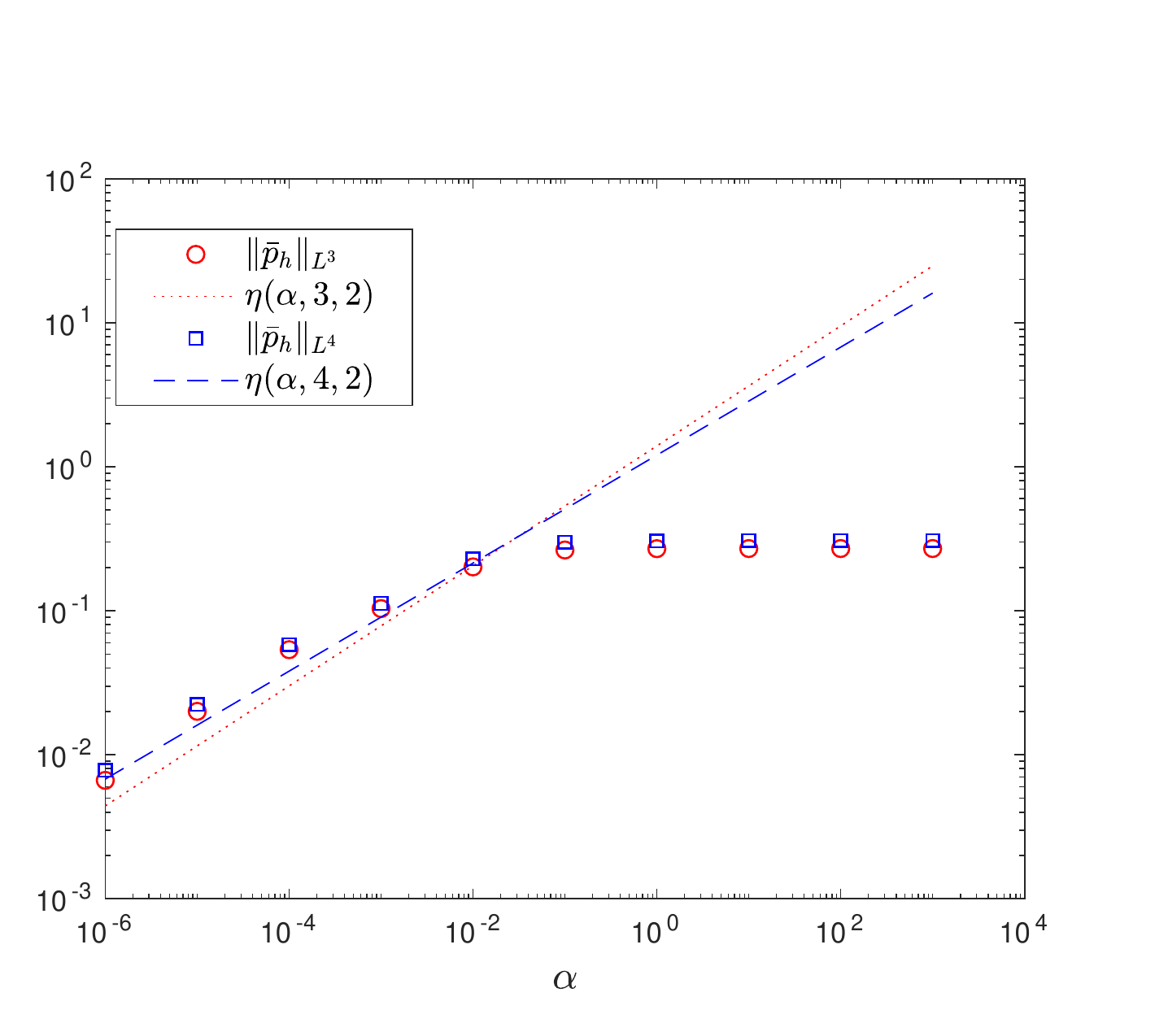}
                \caption{Case 1 with A2}
        \end{subfigure}\\%
        \begin{subfigure}[h!]{0.5\textwidth}
                \includegraphics[trim = 40mm 80mm 30mm 70mm, clip, width=\textwidth]{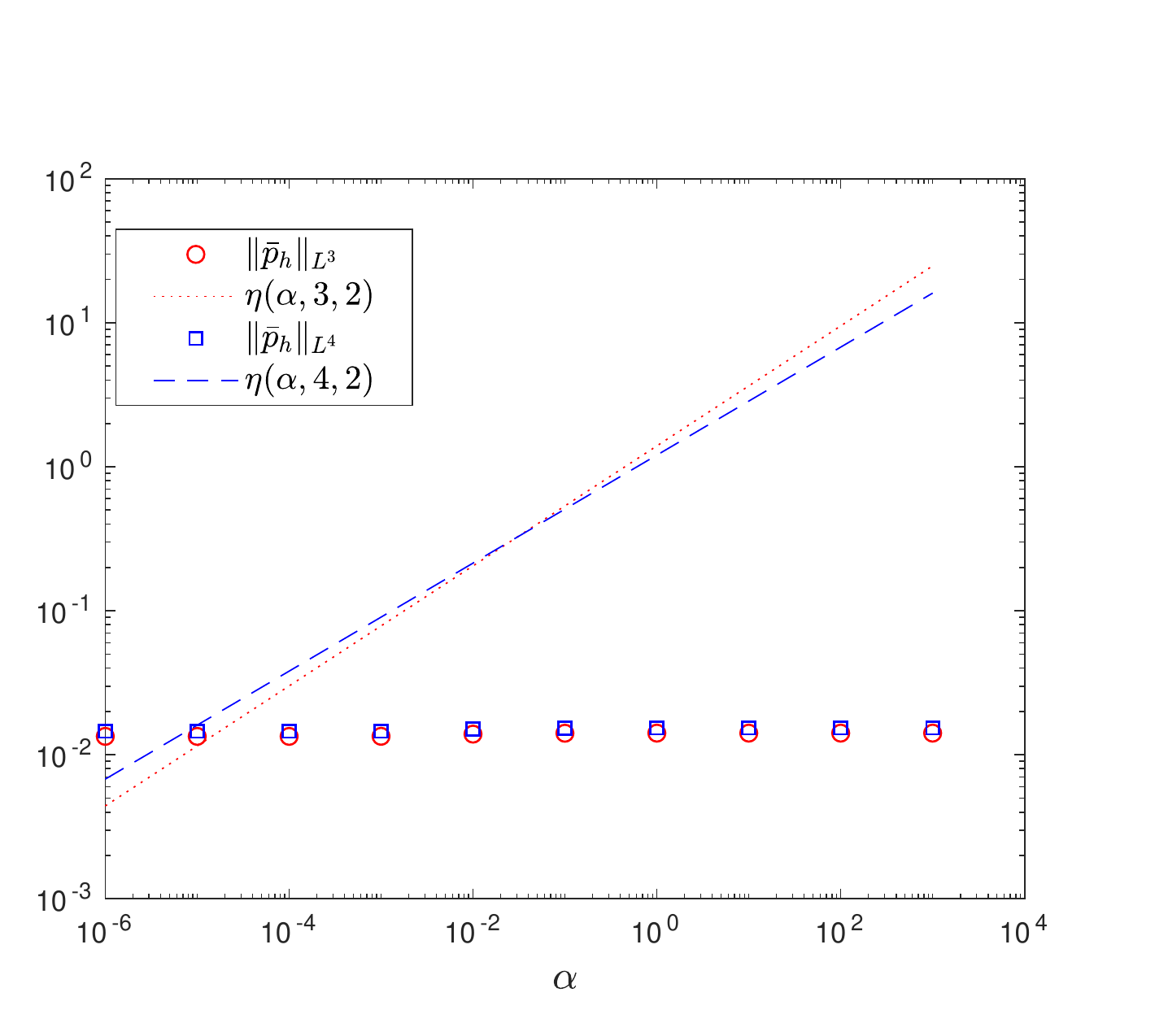}
                \caption{Case 2 with A1}
        \end{subfigure}\hfill%
        \begin{subfigure}[h!]{0.5\textwidth}
                \includegraphics[trim = 40mm 80mm 30mm 70mm, clip, width=\textwidth]{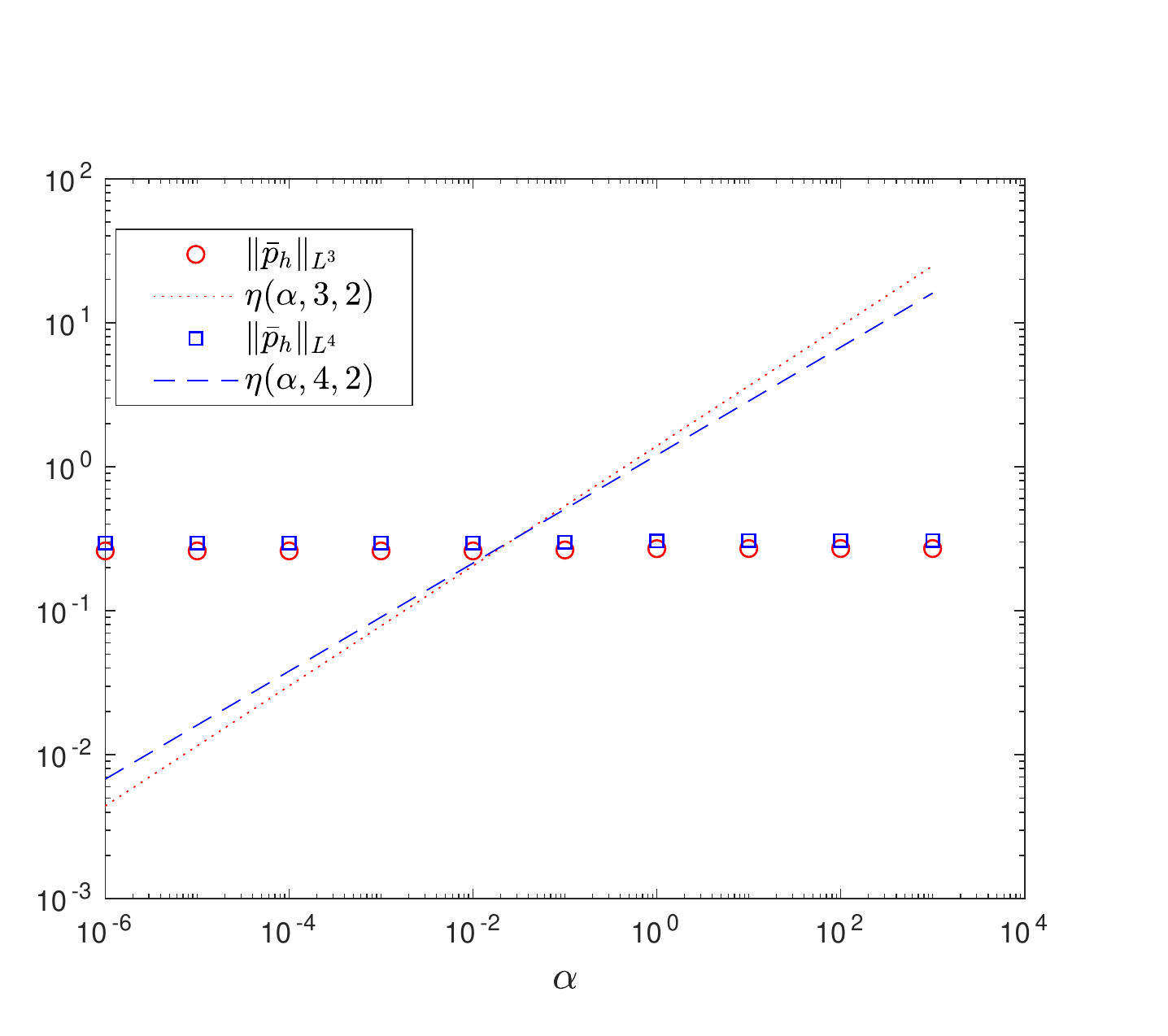}
                \caption{Case 2 with A2}
        \end{subfigure}\\%
        \begin{subfigure}[h!]{0.5\textwidth}
                \includegraphics[trim = 40mm 80mm 30mm 70mm, clip, width=\textwidth]{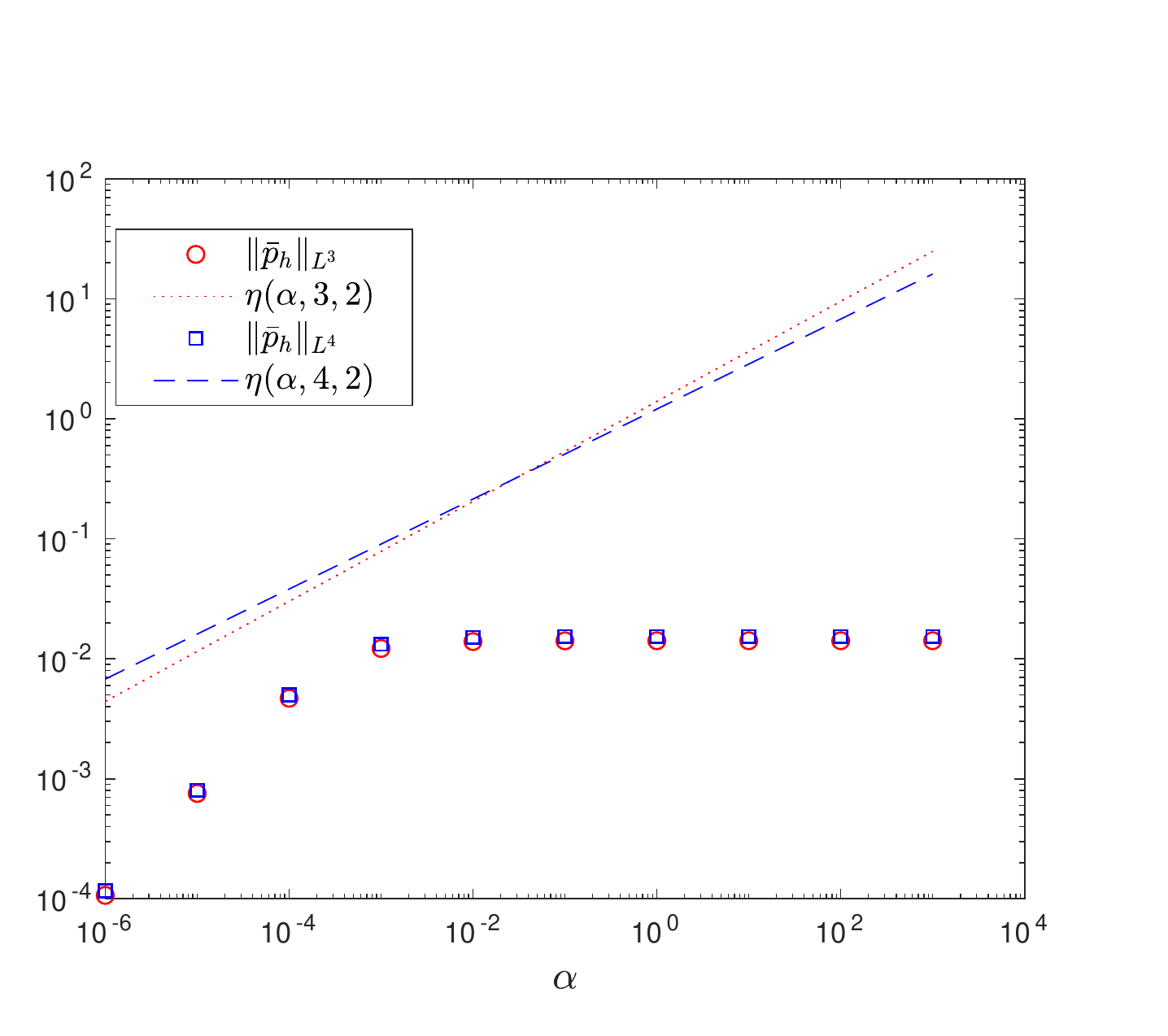}
                \caption{Case 3 with A1}
        \end{subfigure}\hfill%
        \begin{subfigure}[h!]{0.5\textwidth}
                \includegraphics[trim = 40mm 80mm 30mm 70mm, clip, width=\textwidth]{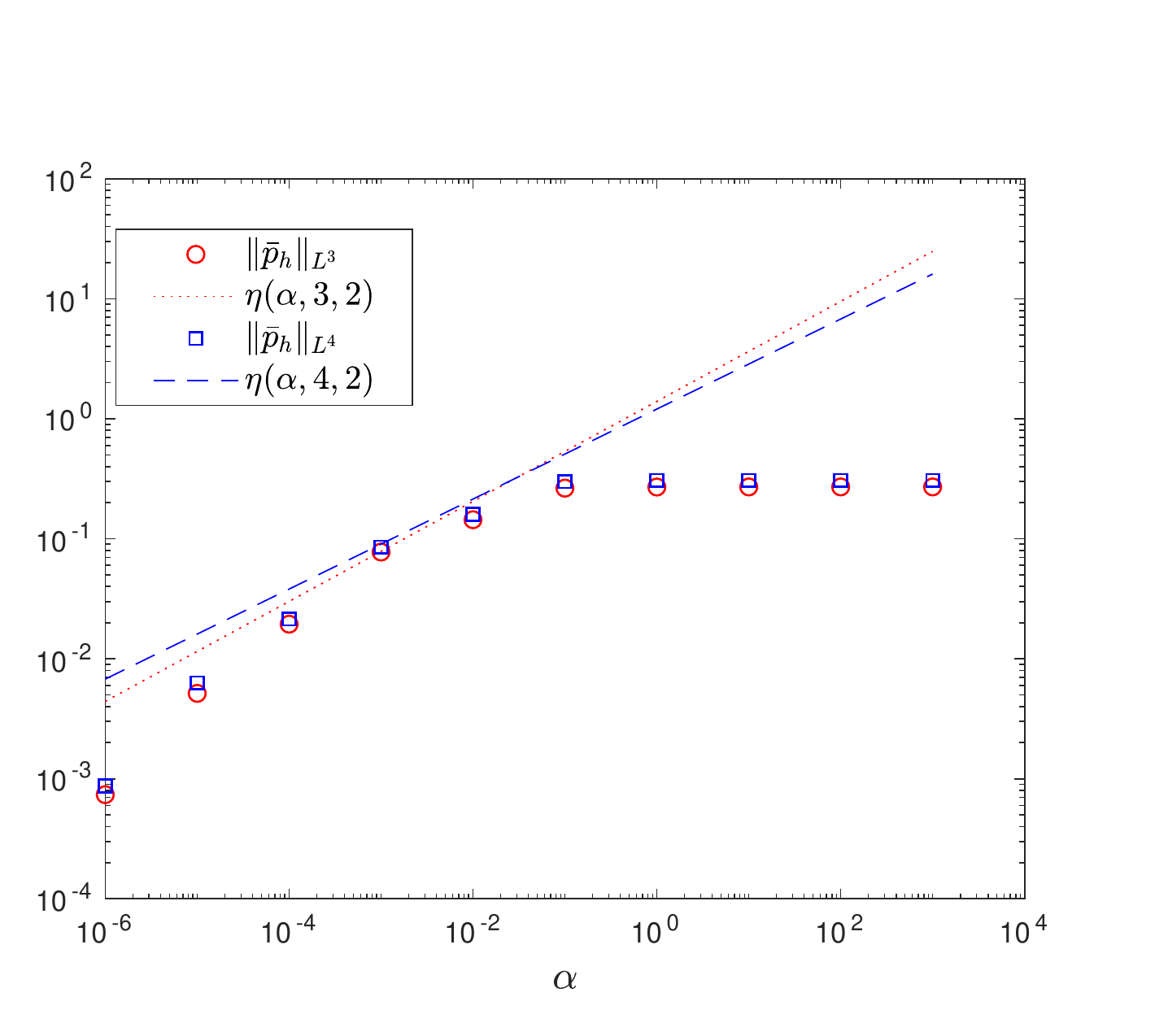}
                \caption{Case 3 with A2}
        \end{subfigure}%
        \caption{Results for $\phi(s)=s^3$}
        \label{F2}
\end{figure}

\begin{figure}[p]
        \centering
        \begin{subfigure}[h!]{0.5\textwidth}
                \includegraphics[trim = 40mm 80mm 30mm 70mm, clip, width=\textwidth]{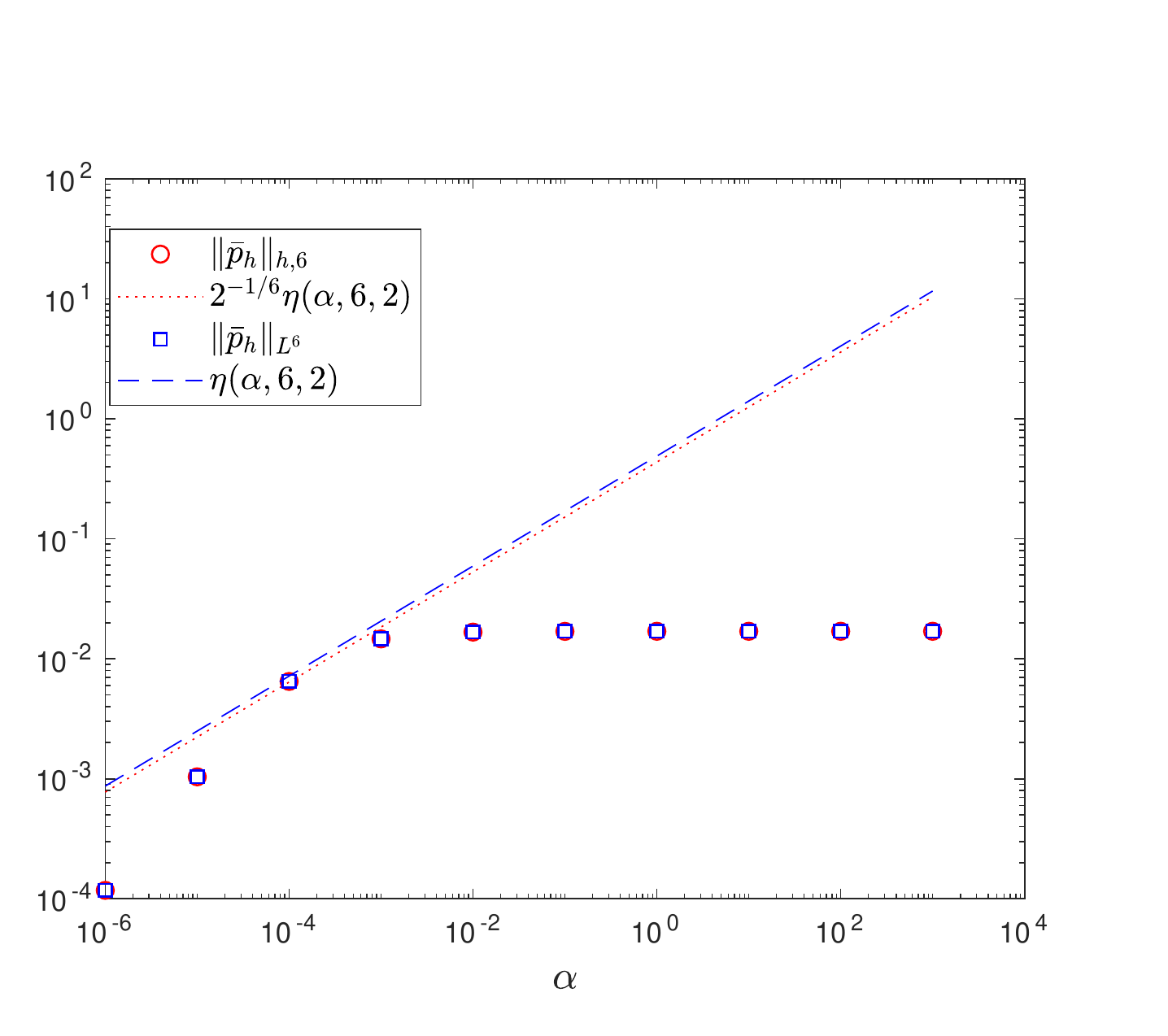}
                \caption{Case 1 with A1}
        \end{subfigure}\hfill%
         \begin{subfigure}[h!]{0.5\textwidth}
                \includegraphics[trim = 40mm 80mm 30mm 70mm, clip, width=\textwidth]{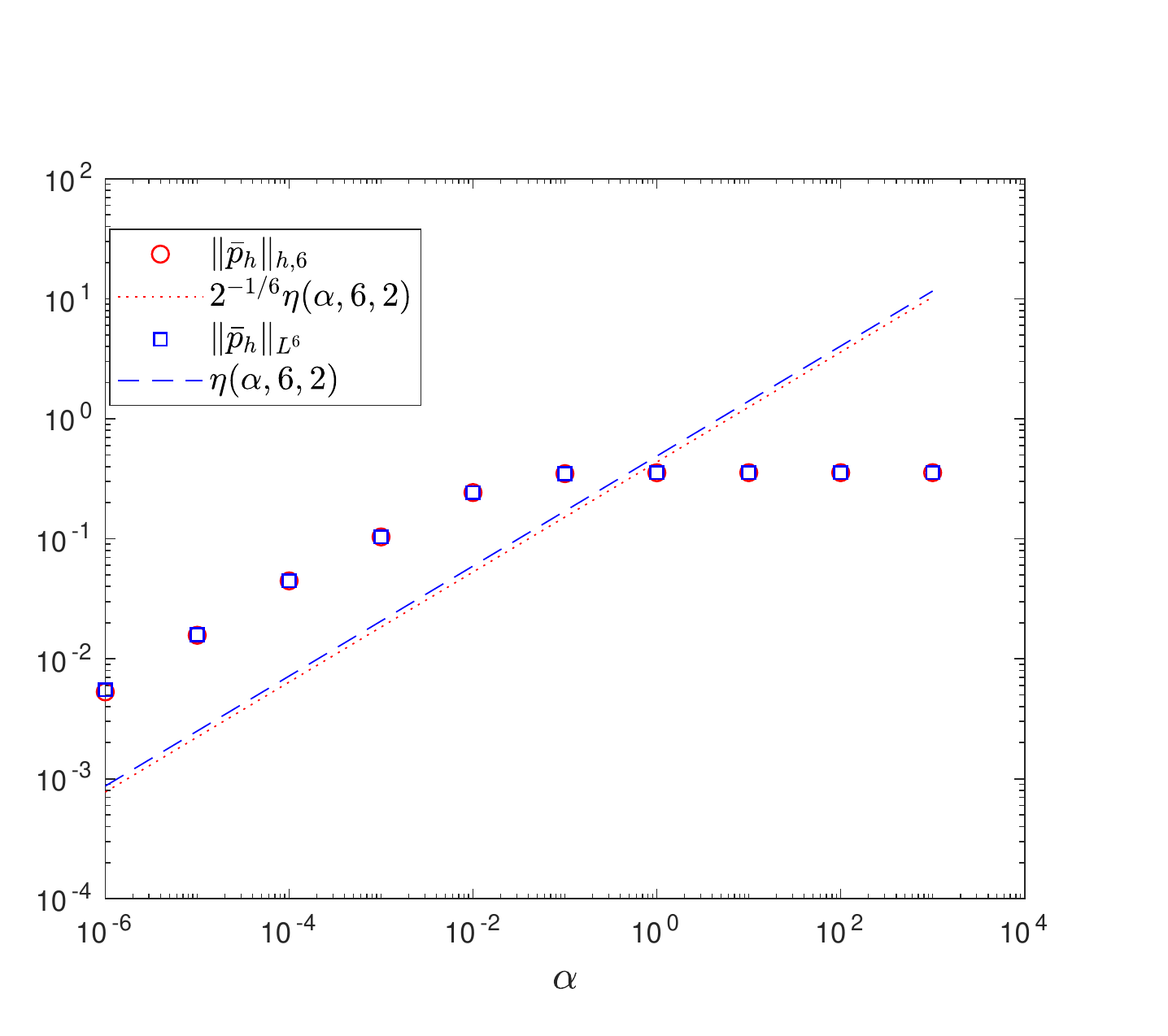}
                \caption{Case 1 with A2}
        \end{subfigure}\\%
        \begin{subfigure}[h!]{0.5\textwidth}
                \includegraphics[trim = 40mm 80mm 30mm 70mm, clip, width=\textwidth]{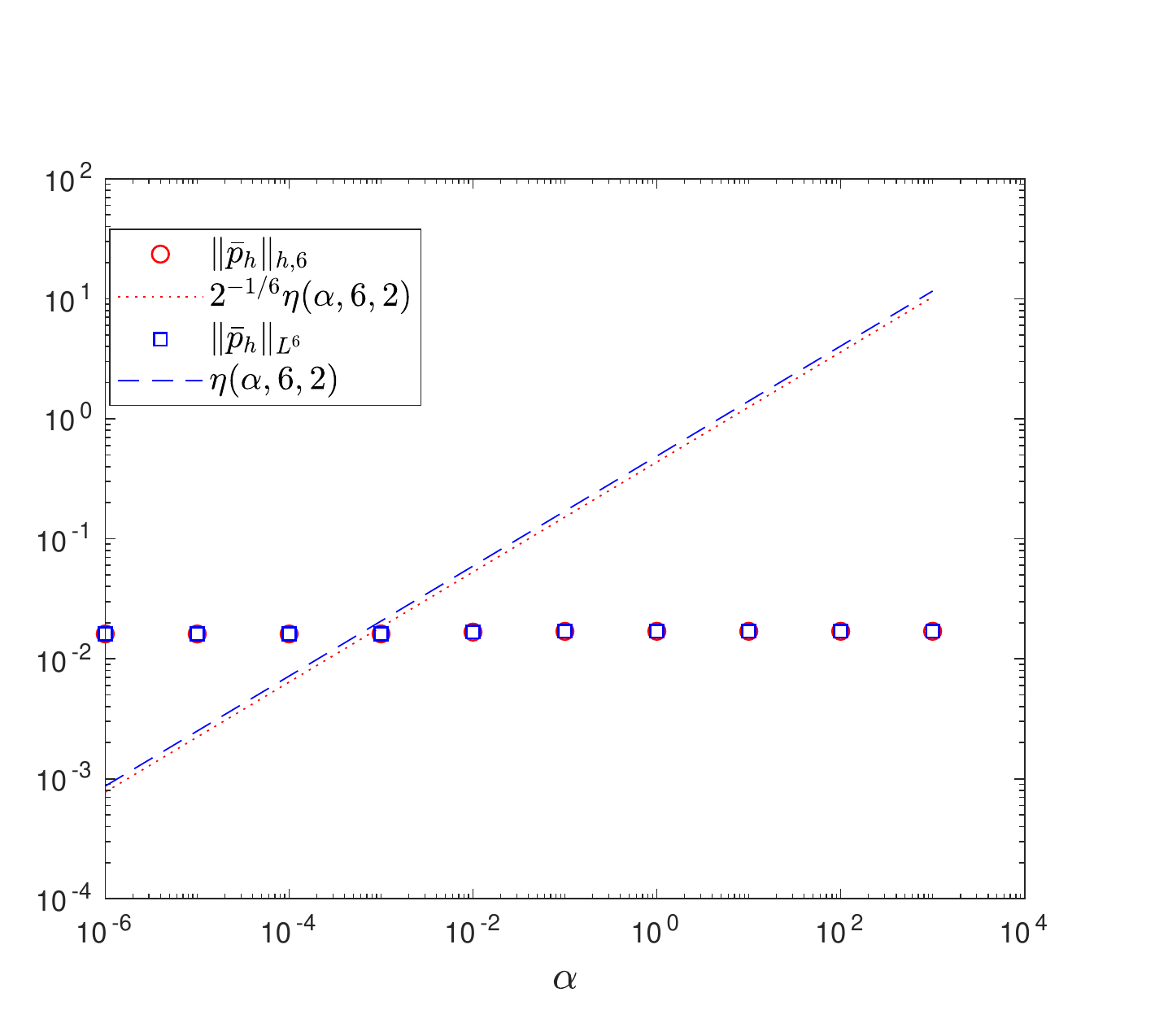}
                \caption{Case 2 with A1}
        \end{subfigure}\hfill%
        \begin{subfigure}[h!]{0.5\textwidth}
                \includegraphics[trim = 40mm 80mm 30mm 70mm, clip, width=\textwidth]{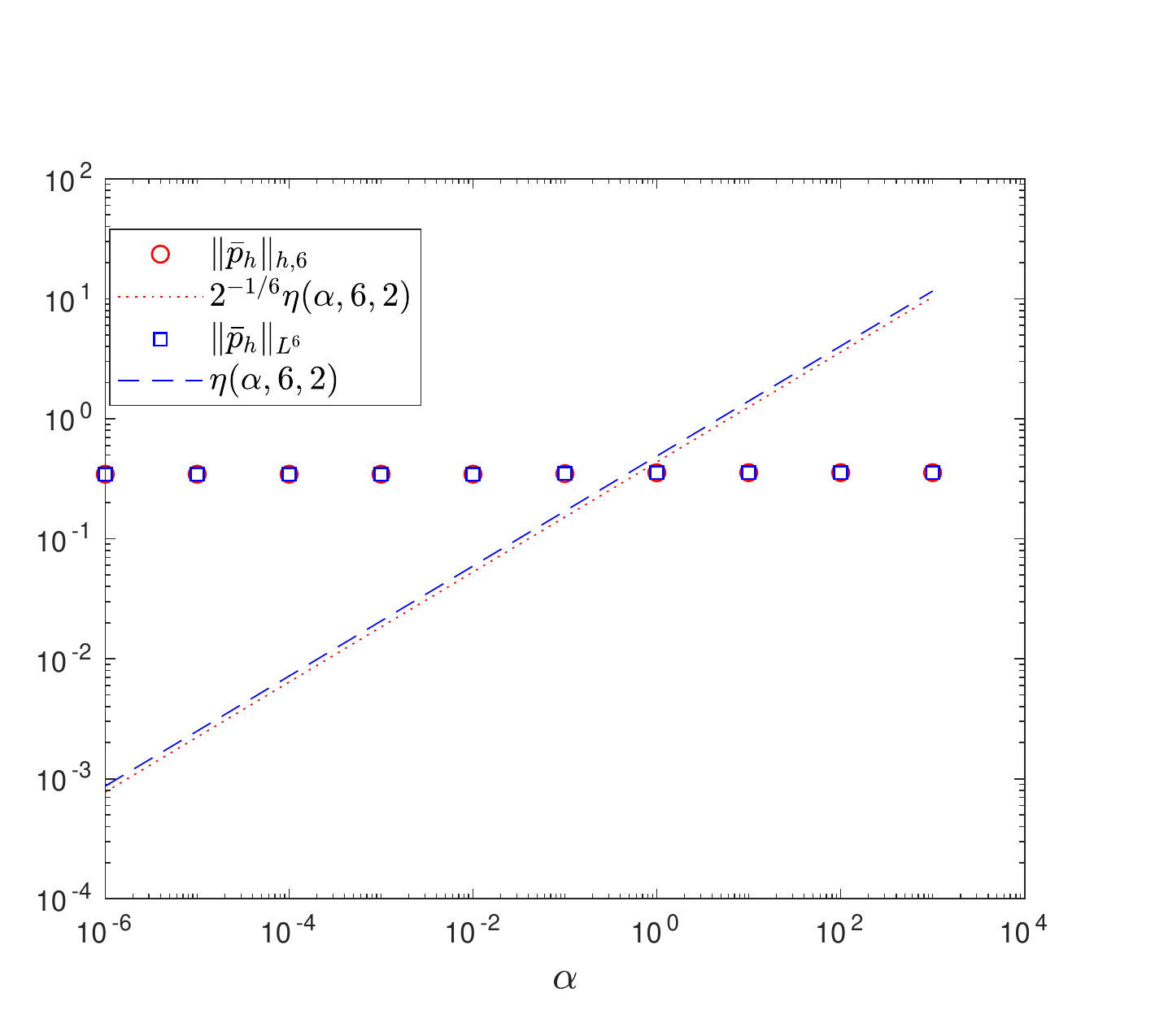}
                \caption{Case 2 with A2}
        \end{subfigure}\\%
        \begin{subfigure}[h!]{0.5\textwidth}
                \includegraphics[trim = 40mm 80mm 30mm 70mm, clip, width=\textwidth]{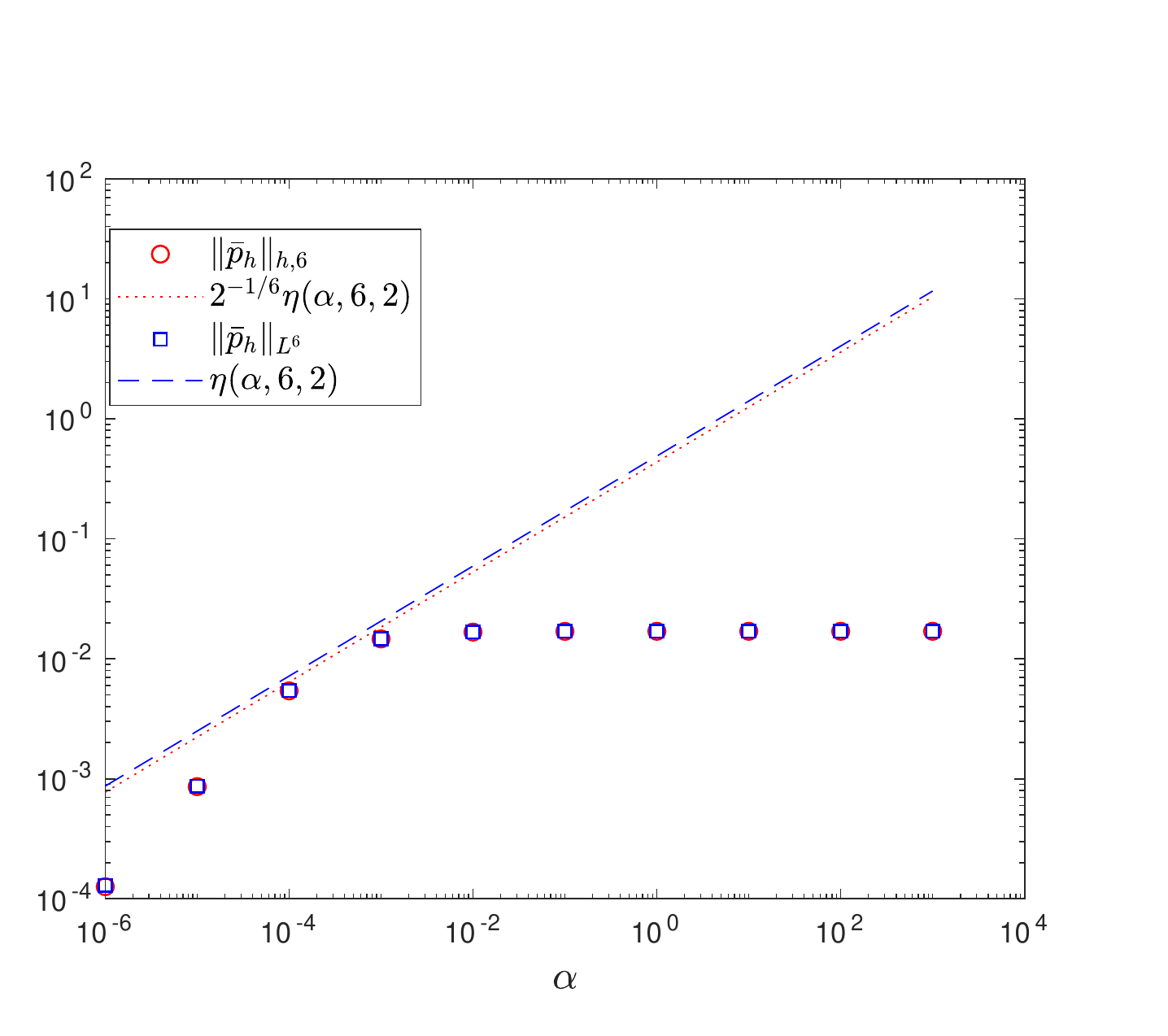}
                \caption{Case 3 with A1}
        \end{subfigure}\hfill%
        \begin{subfigure}[h!]{0.5\textwidth}
                \includegraphics[trim = 40mm 80mm 30mm 70mm, clip, width=\textwidth]{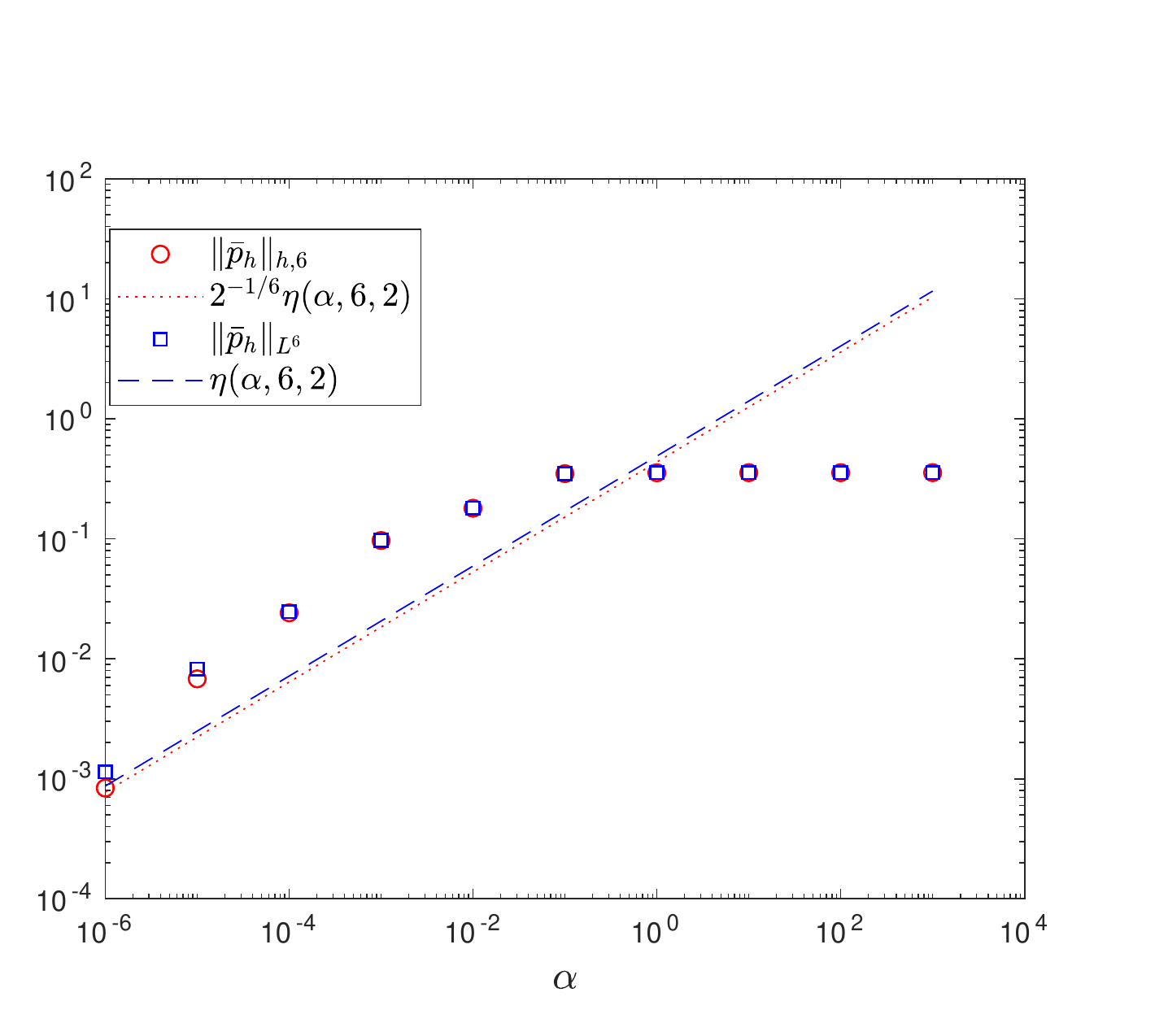}
                \caption{Case 3 with A2}
        \end{subfigure}%
        \caption{Results for $\phi(s)=s^5$}
        \label{F3}
\end{figure}

\setcounter{equation}{0}

\section{Appendix}

\begin{Lemma} \label{intest} Let $d=2$ and  $2 \leq q < \infty$.  Then
\begin{displaymath}
\Vert v_h \Vert_{L^q} \leq \Vert v_h \Vert_{h,q} \leq 4^{\frac{1}{q}} \Vert v_h \Vert_{L^q} \quad \mbox{ for all } v_h \in X_h.
\end{displaymath}
\end{Lemma}
\begin{proof} Let us denote by $\hat T \subset \mathbb{R}^2$ the unit simplex with vertices $\hat a_0=(0,0), \hat a_1=(1,0)$ and $\hat a_2=(0,1)$. Using a scaling argument it is
sufficient to show that 
\begin{equation}  \label{unit}
\int_{\hat T} | p |^q \, d \hat x \leq \int_{\hat T} \hat I_h [ | p |^q] d \hat x  \leq 4 \int_{\hat T} | p |^q \, d \hat x \quad \mbox{ for all } p \in P_1(\hat T),
\end{equation}
where $\hat I_h f=\sum_{j=0}^2 f(\hat a_j) \hat \phi_j$ and $\hat \phi_j(\hat a_i)=\delta_{ij}$. In order to see the first inequality in (\ref{unit}) we
observe that 
\begin{displaymath}
\int_{\hat T} | p |^q \, d \hat x = \int_{\hat T} | \sum_{j=0}^2 p(\hat a_j) \hat \phi_j |^q d \hat x \leq \int_{\hat T} \sum_{j=0}^2 | p(\hat a_j) |^q \hat \phi_j \, d \hat x
= \int_{\hat T} \hat I_h[ | p |^q] d \hat x
\end{displaymath}
in view of the convexity of $t \mapsto | t |^q$ and the properties of $\hat \phi_j, j=0,1,2$. Let us next consider the remaining estimate and first focus on the case $q=2$.
A straightforward calculation shows that
\begin{displaymath}
\int_{\hat T} \hat I_h [ | p |^2] d \hat x = \frac{1}{6} \sum_{j=0}^2 | p(\hat a_j)|^2, \quad \int_{\hat T} | p|^2 \, d \hat x = \frac{1}{24} \sum_{j=0}^2 | p(\hat a_j)|^2 + \frac{1}{24} | p(\frac{\hat a_0+\hat a_1+\hat a_2}{3}) |^2,
\end{displaymath}
which implies that
\begin{equation}  \label{qeq2}
\int_{\hat T} \hat I_h [ | p|^2] d \hat x \leq 4 \int_{\hat T} | p|^2 \, d \hat x.
\end{equation}
Let us introduce the measure $\mu:= \sum_{j=0}^2 m_j \delta_{\hat a_j}$ with $m_j=\int_{\hat T} \hat \phi_j d \hat x= \frac{1}{6}, j=0,1,2$. 
Clearly,
\begin{displaymath}
\Vert p \Vert_{L^q(\mu)}^q:= \int_{\hat T} | p |^q d \mu = \sum_{j=0}^2 | p(\hat a_j)|^q m_j = \int_{\hat T} \hat I_h [ |p|^q] d \hat x.
\end{displaymath}
Now, (\ref{qeq2}) yields that $\Vert p \Vert_{L^2(\mu)} \leq 2 \Vert p \Vert_{L^2(d \hat x)}$, while $\Vert p \Vert_{L^{\infty}(\mu)} \leq \Vert p \Vert_{L^{\infty}(d \hat x)}$, so that
the Riesz--Thorin convexity theorem implies that
\begin{displaymath}
\Vert p \Vert_{L^q(\mu)} \leq 2^{\frac{2}{q}} \Vert p \Vert_{L^q(d \hat x)} \quad \mbox{ for all } p \in P_1(\hat T),
\end{displaymath}
which is (\ref{unit}).
\end{proof}

\providecommand{\href}[2]{#2}
\providecommand{\arxiv}[1]{\href{http://arxiv.org/abs/#1}{arXiv:#1}}
\providecommand{\url}[1]{\texttt{#1}}
\providecommand{\urlprefix}{URL }

\end{document}